\documentclass[11pt]{amsart}

\usepackage{tikz-cd,amsmath,amssymb,amstext,amsgen,amsbsy,amsopn,amsfonts,graphicx,amsthm,semtkzX,mathdots,multicol,mathrsfs}
\usepackage[alphabetic]{amsrefs}
\usepackage{titletoc}
\usepackage[linktocpage=true]{hyperref}
\usepackage{xcolor}
\definecolor{Burgundy1}{RGB}{128,0,32}
\hypersetup{
    colorlinks,
    linkcolor={Burgundy1!90!black},
    citecolor={teal!70!black},
    urlcolor={blue!80!black}
}
\usepackage{cleveref}

\usepackage[left=2.45cm, right=2.45cm, top=3.5cm, bottom=2.5cm,footskip=1cm,headsep=1.5cm]{geometry}
\usepackage[all]{xy}
\usepackage[OT2,T1]{fontenc}

\usepackage{float}

\crefname{equation}{}{}
\numberwithin{equation}{section}

\newtheorem{theorem}[equation]{Theorem}
\newtheorem{lemma}[equation]{Lemma}
\newtheorem{cor}[equation]{Corollary}
\Crefname{cor}{Corollary}{Corollaries}

\newtheorem{proposition}[equation]{Proposition}

\newcommand{\hooklongrightarrow}{\lhook\joinrel\longrightarrow}

\theoremstyle{remark}
\newtheorem{remark}[equation]{Remark}
\theoremstyle{definition}
\newtheorem{example}[equation]{Example}
\theoremstyle{definition}

\theoremstyle{definition}
\newtheorem{defi}[equation]{Definition}
\theoremstyle{definition}
\newtheorem{notation}[equation]{Notation}
\theoremstyle{definition}

\theoremstyle{definition}
\newtheorem{assumption}[equation]{Assumption}
\theoremstyle{definition}

\DeclareSymbolFont{cyrletters}{OT2}{wncyr}{m}{n}
\DeclareMathSymbol{\Sha}{\mathalpha}{cyrletters}{"58}

\address{School of Mathematics and Statistics, University of Glasgow, University Place, Glasgow, G12 8QQ.}
\email{ajmorgan44@gmail.com}

\begin{document}


\title{Hasse Principle for Kummer Varieties in the case of generic $2$-torsion}

\author{Adam Morgan}
\setcounter{tocdepth}{1}
\maketitle

\begin{abstract}
Conditional on finiteness of relevant Shafarevich--Tate groups, Harpaz and
Skorobogatov used Swinnerton-Dyer's descent-fibration method to establish the Hasse principle for Kummer varieties associated to a 2-covering of a principally polarised abelian variety  under certain largeness assumptions on its mod $2$ Galois image.  Their method breaks down, however,  when 
the Galois image is maximal, due to the possible failure of the Shafarevich--Tate
group of quadratic twists of $A$ to have square order. In this work we overcome this obstruction by combining second descent ideas in the spirit of Harpaz and Smith with results on the parity of $2^\infty$-Selmer ranks in quadratic twist families. This allows Swinnerton-Dyer's method to be successfully applied to K3 surfaces arising as quotients of $2$-coverings of Jacobians of genus $2$ curves with no rational Weierstrass points. 
\end{abstract}

\tableofcontents
\vspace{-10pt}

\section{Introduction}

The aim of this article is to extend the scope of a method developed in works of Swinnerton-Dyer, Harpaz and Skorobogatov \cite{SkoroSwinDyer2005,HS16,Har19}, that uses the behaviour of $2$-Selmer groups in quadratic twist families to study rational points on certain K3 surfaces over number fields. More generally, the method applies to  \textit{generalised Kummer varieties} arising as the  minimal desingularisation of the quotient of a $2$-covering of a principally polarised abelian variety by its antipodal involution. The method has its origins in work of Swinnerton-Dyer \cite{SwinDyer1995}, and is often referred to as the \textit{descent-fibration} method. For related works in a slightly different setting, including that of $K3$ surfaces fibred in genus $1$ curves, see \cite{skoroswin98,SwinDyer2000,SwinDyer2001,witt07,Har19b} and references therein.

Swinnerton-Dyer's method is currently the only known method able to say something about the existence of rational points on K3 surfaces, albeit conditional on finiteness of Shafarevich--Tate groups (and sometimes Schinzel's hypothesis). In particular, it provides some of the only evidence for the conjecture that the Brauer--Manin obstruction is the only obstruction to the Hasse principle for such surfaces; cf. \cite[p. 77]{Skoro2009} and \cite[p. 484]{SkoroZarhin2008}.

In line with most existing results on the behaviour of $2$-Selmer groups in quadratic twist families (a notable exception being the recent work of Smith \cite{smith2022distribution1,smith2022distribution2}), the cited works split broadly into $2$ categories. The papers \cite{SkoroSwinDyer2005,Har19} consider the case where all $2$-torsion points of the abelian variety are defined over the base field. By contrast, the work \cite{HS16}, which builds on Selmer group techniques developed by Mazur and Rubin \cite{MazRub2010}, has various large image assumptions on the Galois action on the $2$-torsion. 

The present work fits into the latter category, and treats the case of principally polarised abelian varieties where the Galois action on the $2$-torsion is maximal, i.e. equal to the full sympletic group of automorphisms preserving the Weil pairing. This is made possible by the removal of an assumption, present in all previous iterations of the method, that the principal polarisation be induced by a symmetric line bundle. The relevance of this geometric condition is that it forces the Cassels--Tate pairing on the Shafarevich--Tate groups of all quadratic twists of the abelian variety to be alternating. As shown by Poonen--Stoll  \cite{MR1740984}, this need not hold in general.

Accounting for the possible failure of the Cassels--Tate pairing to be alternating necessitates significant changes to the arguments. We discuss the relevant adaptations in \Cref{sec:adaptations}, but mention that, compared to \cite{HS16}, the key new features of our work  consist of a version of the second descent step introduced by Harpaz \cite{Har19} in the case of full rational $2$-torsion, and the incorporation of results on the parity of $2^\infty$-Selmer ranks in quadratic twist families building on my previous work \cite{MR3951582}. 

The distinction between the cases we can handle and those treated in the work of Skorobogatov--Harpaz is most readily visible when restricted to Jacobians of genus $2$ curves. Here for the first time, our work allows Swinnerton-Dyer's method to be successfully applied to Jacobians of curves with no rational Weierstrass point. We will discuss this case now, postponing a full statement of the main result to \Cref{subset_main_results}.  Since the case of genus $2$ Jacobians is the one pertaining to K3 surfaces, this is already one of the most important instances of our work. 

For simplicitly we take $\mathbb{Q}$ as our base-field, though an analagous statement remains true over an arbitrary number field  (see \Cref{thm:hyp_curves_main}). A genus $2$ curve $C$ over $\mathbb{Q}$ is necessarily hyperelliptic, and can be represented by a Weierstrass equation
\begin{equation} \label{weierstrass equation}
C:y^2=f(x), 
\end{equation}
where $f(x)\in \mathbb{Z}[x]$ is a squarefree polynomial of degree $6$. Such a curve has a rational Weierstrass point if and only if $f(x)$ has a rational root, in which case a change of variables can be used to bring the equation \eqref{weierstrass equation} into a form where the defining polynomial has degree $5$ instead. The $2$-division field of the Jacobian $J$ of $C$ coincides with the splitting field of $f(x)$, hence is naturally viewed as a subgroup of the symmetric group $S_6$. Our main result, \Cref{main_thm}, implies the following. In the statement, $H^1(\mathbb{Q},J[2])$ denotes the first Galois cohomology group of the Galois module of $2$-torsion points in $J$; it parametrises isomorphism classes of $2$-coverings of the Jacobian of $C$.

\begin{theorem} \label{Theorem_degree_6}
Suppose that the Galois group of $f(x)$ is isomorphic to $S_6$.   Let $\mathfrak{a}\in  H^1(\mathbb{Q},J[2])$ and suppose there is an odd prime $p$ such that $\mathfrak{a}$ is unramified at $p$, the leading coefficient of $f(x)$ is coprime to $p$, and the discriminant of $f(x)$ is divisible by $p$ exactly once. 
Assume that the $2$-primary part of the Shafarevich--Tate group of every quadratic twist of $J$ is finite. 

Then the generalised Kummer surface $X$ associated to $\mathfrak{a}$ satisfies the Hasse principle. Moreover, if $X$ is everywhere locally soluble and $\mathfrak{a}$ has non-trivial restriction to the splitting field of $f(x)$,  then $X$ has a Zariski-dense set of $\mathbb{Q}$-points.\footnote{For an explicit description of the Kummer K3 surface associated to $\mathfrak{a}$ as the intersection of three quadrics in $\mathbb{P}^5$, as well as explicit conditions ensuring that the class $\mathfrak{a}$ is unramified at $p$, see \Cref{Kummer_intersection_of_quadrics}.}   
\end{theorem}

By contrast, \cite[Theorem B]{HS16} establishes essentially the same result in the case that $f(x)$ has degree $5$ and Galois group $S_5$ (in the presence of other assumptions there, the condition that the polynomial be irreducible in that result is equivalent to it having Galois group $S_5$; see the proof of Theorem B in \cite{HS16}). 

The conditions in \Cref{Theorem_degree_6} apply to reasonably generic genus $2$ curves. Indeed, as above,  any such curve can be represented by an equation of the form \eqref{weierstrass equation}. Moreover, a result of van der Waerden implies that, when ordered by the maximum of the absolute values of the coefficients, $100\%$ of integer sextic polynomials have Galois group $S_6$ (see recent work of Bhargava \cite{Bhargava2021} for the strongest statement). Per work of Bhargava--Shankar--Wang \cite{Bhargava2022}, for example, the existence of an odd prime $p$ as in the statement is also not particularly restrictive. 

Before describing the results of the paper in more detail, we address a point regarding \Cref{Theorem_degree_6}. In light of the discussion before its statement, one might wonder why the conclusion is that the Kummer surface  satisfy the Hasse principle, with no mention of the Brauer--Manin obstruction. This is also the case in \cite{HS16}. There, the explanation is found in work of Skorobogatov--Zarhin \cite{SkoroZarhin2017}, showing (unconditionally) that the assumptions on the abelian variety force the relevant part of the Brauer group to vanish. It seems plausible that similar results hold true here. We remark also that, by contrast, the Brauer--Manin obstruction \textit{is} seen to intervene in the work of Harpaz \cite{Har19}. 

\subsection{Statement of the main result} \label{subset_main_results}
We now describe in detail the main results of the paper.
Let $A$ be a principally polarised abelian variety of dimension $d\geq 2$, defined over a number field $K$. Let $L=K(A[2])$ and let $G=\textup{Gal}(L/K)$,  viewed as a subgroup of automorphisms of $A[2]$ preserving the Weil pairing (with respect to the given principal polarisation).   Denote by $\mathfrak{c}\in H^1(K,A[2])$ the torsor parametrising symmetric line bundles inducing the given principal polarisation (see \Cref{ssec:class_c}); equivalently, $\mathfrak{c}$ is the cohomology class associated to the $G_K$-set of quadratic refinements of the Weil pairing on $A[2]$. It arises via inflation from $H^1(G,A[2])$. For each $g\in G$, denote by $\left \langle g \right \rangle$ the subgroup of $G$ generated by $g$.

\begin{assumption} \label{existence_of_frob_elements_prop_1}
We make the following assumptions on $G$:
\begin{itemize}

\item[(A)] $A[2]$ is a simple $\mathbb{F}_2[G]$-module and $\textup{End}_G(A[2])=\mathbb{F}_2,$
\item[(B)] there exists an element $g\in G$ such that $\dim A[2]^{g}=0$,
\item[(C)] the group 
\begin{equation} \label{intersection_c_contains}
\bigcap_{g\in G}\ker\Big(H^1(G,A[2])\stackrel{\textup{res}}{\longrightarrow}  H^1(\left \langle g\right \rangle, A[2]) \Big) 
\end{equation}
 is generated by the class $\mathfrak{c}$ (which we allow to be $0$).
\end{itemize}
\end{assumption}

\begin{remark}
It is always the case that the class $\mathfrak{c}$ lies in the intersection appearing in \eqref{intersection_c_contains}; see \Cref{ssec:class_c}.
\end{remark}

 \begin{example} \label{ex:2-tors}
All parts of \Cref{existence_of_frob_elements_prop_1} are  satisfied when: 
\begin{itemize}
\item[(i)] $G$ is the full group $\textup{Sp}(A[2])$ of symplectic automorphisms of $A[2]$ preserving the Weil pairing (hence isomorphic to $\textup{Sp}_{2d}(\mathbb{F}_2)$) or,
\item[(ii)]   $A$ is the canonically principally polarised Jacobian of a hyperelliptic curve $y^2=f(x)$, where  $f(x)\in K[x]$ is a squarefree polynomial of even degree $2d+2\geq 6,$ such that $f(x)$ has Galois group isomorphic to the symmetric group $S_{2d+2}$.
\end{itemize}
Indeed, in case (i) parts (A) and (B) are easy to check. For part (C) when $G\cong \textup{Sp}_{2d}(\mathbb{F}_2)$ (which includes the $d=2$ case of (ii)), see work of Pollatsek \cite[Section 4]{Pol71}  which shows that the group  $H^1(G,A[2])$ is isomorphic to $\mathbb{F}_2$, generated by the class $\mathfrak{c}$. For case (ii) when $d\geq 3$, see \Cref{sec:specialising_to_hyp_curves}.
\end{example}

\begin{notation} 
Fix a class $\mathfrak{a}$ in $H^1(K,A[2])$. We denote by $Y$ the  $2$-covering of $A$ corresponding to $\mathfrak{a}$, and denote by $X=\textup{Kum}(Y)$ the corresponding generalised Kummer variety, defined as in \cite[Section 6]{HS16}.
\end{notation}

\begin{assumption}\label{assumption:local_solubility}
Assume that $\mathfrak{a}$ is not in $H^1(G,A[2])$ (equivalently, $\mathfrak{a}$ has non-trivial restriction to $H^1(L,A[2])$). Assume also that there is an odd prime $\mathfrak{p}_0$ of $K$ such that: $\textup{res}_{\mathfrak{p}_0}(\mathfrak{a})$ is unramified,  $A$ has semistable reduction and odd conductor exponent at $\mathfrak{p}_0$, and the (group of geometric points of the) N\'{e}ron component group of $A$ over $K_{\mathfrak{p}_0}$ has odd order. 
\end{assumption}


Our main result is the following: 

\begin{theorem} \label{main_thm}
Assume that the $2$-primary part of the Shafarevich--Tate group of every quadratic twist of $A$ with $2^\infty$-Selmer rank $1$ is finite. Under Assumptions \ref{existence_of_frob_elements_prop_1} and \ref{assumption:local_solubility} above, if  $X(K_v)\neq \emptyset$ for all places $v$ of $K$, then  $X$ has a Zariski-dense set of $K$-points. In particular, $X$ satisfies the Hasse principle. 
\end{theorem}

\begin{remark} \label{rem_g_1_fix}
The case when, in addition to \Cref{existence_of_frob_elements_prop_1}, we have  $H^1(G,A[2])=0$ (hence $\mathfrak{c}=0$) and an element $g\in G$ for which $\dim A[2]^g=1$, is treated in \cite[Theorem 2.3]{HS16}. These assumptions preclude both cases (i) and (ii) of \Cref{ex:2-tors} above. Even when $H^1(G,A[2])=0$, it can happen that no $g\in G$ has  $\dim A[2]^g=1$, as is the case for Jacobians of odd degree hyperelliptic curves $y^2=f(x)$ where $f(x)$ has degree $2d+1\geq 5$ and Galois group the alternating group $A_{2d+1}$. 

We remark however that \cite{HS16} considers moreover  products of abelian varieties which individually satisfy suitable large Galois image conditions with respect to their $2$-torsion. We have elected not to consider this possible extension in our work.
\end{remark}

\begin{remark} \label{rem:dim_2_subgroups} 
When $\dim A=2$ we have $\textup{Sp}_{4}(\mathbb{F}_2)\cong S_6$. There are, up to conjugacy in $S_6$, $9$ subgroups for which   \Cref{existence_of_frob_elements_prop_1} holds. Of these, $S_6$, $A_6$, and $S_5$ embedded transitively as $T6.13$, have $\mathfrak{c}\neq 0$. Here we are using the notation of \cite[Section 2.9]{DM96}  to describe permutation groups. The remaining transitive subgroups are $T6.9$, $T6.10$ and $T6.11$. The list is completed by the subgroups $S_5$, $A_5$, and $T5.3$. By contrast, the subgroups satisfying \cite[Assumptions (a) - (d)]{HS16} are the transitive subgroups $T6.9$ and $T6.10$, and the non-transitive subgroups $S_5$ and $T5.3$.
\end{remark}

When $A$ is the Jacobian of a hyperelliptic curve,  \Cref{main_thm} implies the following result, generalising \Cref{Theorem_degree_6}.  

\begin{cor} \label{thm:hyp_curves_main}
Let $d\geq 2$ and let $f(x)\in K[x]$ be a squarefree polynomial of degree $2d+2$ and Galois group $S_{2d+2}$. Let $A$ be the Jacobian of the hyperelliptic curve $y^2=f(x)$ and let $\mathfrak{a} \in  H^1(K,A[2])$. 
Suppose  there is an odd prime $\mathfrak{p}_0$ of $K$ such that $\mathfrak{a}$ is unramified at $\mathfrak{p}_0$,  $f(x)$ is integral at $\mathfrak{p}_0$, and  such that
\[\textup{ord}_{\mathfrak{p}_0}\textup{disc}(f)=1\quad \textup{ and }\quad\textup{ord}_{\mathfrak{p}_0}(c_f)=0,\]   
where $\textup{disc}(f)$ is the discriminant of $f(x)$ and $c_f$ is the leading coefficient of $f(x)$. 
Assume that the $2$-primary part of the Shafarevich--Tate group of every quadratic twist of $A$ with $2^\infty$-Selmer rank  $1$ is finite.  Then the generalised Kummer variety  associated to $\mathfrak{a}$  satisfies the Hasse principle.  
\end{cor}

\begin{remark}
In the setting of \Cref{thm:hyp_curves_main} above, if $\mathfrak{a}\notin H^1(G,A[2])$ and the Kummer variety is everywhere locally soluble  then (analogously to the conclusion of \Cref{main_thm}) it has  a Zariski-dense set of $K$-points. When $\mathfrak{a}\in H^1(G,A[2])$, an easy argument shows that the Kummer variety always has a $K$-point, hence at least satisfies the Hasse principle. See \Cref{sec:specialising_to_hyp_curves} for details.
\end{remark}

%
%
%

\subsection{Outline of the argument} \label{sec:adaptations}
As mentioned previously, the method we use to prove \Cref{main_thm}   goes back to Swinnerton-Dyer, with the closest counterpart to our work being that of   Harpaz--Skorobogatov \cite{HS16}.  Assuming for simplicity finiteness of all Shararevich--Tate groups, the strategy employed in that work to prove an analogue of \Cref{main_thm} is roughly as follows:
\begin{itemize}
\item[(1)] Assuming local solubility of $X$, one uses the fibration method to establish the existence of a quadratic character $\chi_0$ such that $\mathfrak{a}$ lies in the $2$-Selmer group $\textup{Sel}^2(A^{\chi_0}/K)$ of the quadratic twist $A^{\chi_0}$.  Here we view $\textup{Sel}^2(A^{\chi_0}/K)$ as a subgroup of $H^1(K,A[2])$ via the canonical identification $A[2]\cong A^{\chi_0}[2]$ of $G_K$-modules.
\item[(2)] Using a variant of a twisting procedure due to Mazur and Rubin \cite{MazRub2010}, one successively alters $\chi_0$ to find a $\chi_1$ such that $\textup{Sel}^2(A^{\chi_1}/K)$ is $1$-dimensional and generated by $\mathfrak{a}$.
\item[(3)] Under the assumptions of  \cite{HS16}, the Cassels--Tate pairing on $\Sha(A^{\chi_1}/K)[2^\infty]$ is (non-degenerate and) alternating by work of Poonen--Stoll  \cite{MR1740984}, forcing $\dim \Sha(A^{\chi_1}/K)[2]$ to be even. Since $\dim \textup{Sel}^2(A^{\chi_1}/K)=1$ we necessarily find that $\Sha(A^{\chi_1}/K)[2]=0$, and conclude that $\mathfrak{a}$ maps to $0$ in $H^1(K,A^{\chi_1})$. As explained in \cite[Section 6]{HS16}, this is sufficient to show that $X$ has a $K$-point, and in fact a Zariski dense set of such. 
\end{itemize}

In our setting there are issues in pulling off steps $(2)$ and $(3)$  due  to the possible non-triviality of the class $\mathfrak{c}$. Firstly, this is precisely the class shown by Poonen and Stoll to measure the obstruction to the Cassels--Tate pairing being alternating. So for quadratic characters $\chi$, it is no longer possible to conclude that $\Sha(A^{\chi}/K)[2]=0$ from the fact that $\dim\textup{Sel}^2(A^{\chi}/K)=1$  alone. Moreover, one can show  that $\mathfrak{c}$ lies in $\textup{Sel}^2(A^{\chi}/K)$ for \textit{all} quadratic characters $\chi$ (cf. \Cref{ssec:class_c}), so the most one could hope to establish via an analogue of (2) is the existence of a character $\chi_1$ such that $\textup{Sel}^2(A^{\chi_1}/K)$ is generated by $\mathfrak{a}$ and  $\mathfrak{c}$. To overcome these issues we incorporate second descent ideas, in the spirit of Harpaz   \cite{Har19} and Smith \cite{smi16}, and study the variation of the Cassels--Tate pairing on $2$-Selmer groups under quadratic twist.  Actually, instead of first twisting to make the $2$-Selmer group as small as possible before moving onto the second descent step  (as is done in  \cite{Har19}), we find it more efficient to simply abandon the first descent step and pass straight to the second; this is what allows us to remove Assumption (C) of \cite{HS16}.   Our strategy is thus as follows:
\begin{itemize}
\item As  in \cite{HS16}, we first show that local solubility of $X$ implies the existence of a quadratic character $\chi_0$ such that $\mathfrak{a}$ lies in $\textup{Sel}^2(A^{\chi_0}/K)$. In addition, by extending results of \cite{MR3951582} concerning the behaviour of parities of $2^\infty$-Selmer ranks in quadratic twist families, we can (crucially) also assume that $A^{\chi_0}$ has odd $2^\infty$-Selmer rank; see \Cref{parity section} for more details. It is this last step that requires \Cref{assumption:local_solubility}.
\item Using a variant of an observation of Mazur--Rubin,  we give a method for producing many quadratic twists all having odd $2^\infty$-Selmer rank and identical $2$-Selmer group to that of $A^{\chi_0}$ (see \Cref{selmer_agreement_lemma} and  \Cref{approximating_extensions_lemma}). Our approach to this step uses \Cref{existence_of_frob_elements_prop_1} (B). 
\item We then show that, as we vary over twists as in the previous step, the resulting Cassels--Tate pairings on the common $2$-Selmer group vary over all possible pairings subject to certain forced structure; this is the key technical result of the paper, and is described in detail below.  This step crucially uses part (A) and (B) of  \Cref{existence_of_frob_elements_prop_1}. Having done this, it is easy to find such a pairing having only $\mathfrak{a}$ in its kernel. For a corresponding twist $\chi$, $\mathfrak{a}$ is the unique element lifting to the $4$-Selmer group of $A^\chi$. Since $A^\chi$ has odd  $2^\infty$-Selmer rank by construction, the only possibility, assuming finiteness of $\Sha(A^\chi/K)[2^\infty]$, is that $\mathfrak{a}$ has trivial image in $\Sha(A^\chi/K)[2^\infty]$. We can then conclude as before that the Kummer variety associated to $\mathfrak{a}$ has a Zariski-dense set of $K$-points. 
\end{itemize}

\subsection{Variation of the Cassels--Tate pairing under quadratic twist}
As mentioned, a key step in the proof of \Cref{main_thm}  shows that we can twist by quadratic characters in such a way that the $2$-Selmer group is preserved, but the Cassels--Tate pairing on the $2$-Selmer group takes any shape we desire, subject only to certain structural constraints. 

To explain the result, let $A$ be a principally polarised abelian variety satisfying parts (A) and (B) of \Cref{existence_of_frob_elements_prop_1}. Call a $\tfrac{1}{2}\mathbb{Z}/\mathbb{Z}$-valued bilinear pairing $P$ on $\textup{Sel}^2(A/K)$   \textit{admissible} if
\[\phantom{move accr} P(\mathfrak{a},\mathfrak{a})=P(\mathfrak{a},\mathfrak{c}) \quad \quad\textup{ for all }\mathfrak{a}\in \textup{Sel}^2(A/K),\]  
and 
\[2P(\mathfrak{c},\mathfrak{c})\equiv  \dim \textup{Sel}^2(A/K)+\textup{rk}_2(A/K)\quad (\textup{mod } 2\mathbb{Z}),\]
where $\textup{rk}_2(A/K)$ denotes the $2^\infty$-Selmer rank of $A$. 
If the class $\mathfrak{c}$ is zero, this amounts to saying that $P$ is alternating and that the $2^\infty$-Selmer rank of $A$ is congruent modulo $2$ to the dimension of its $2$-Selmer group. 

It follows from work of Poonen--Stoll \cite[Theorem 8]{MR1740984} that the Cassels--Tate pairing on $\textup{Sel}^2(A/K)$ is admissible (cf. \Cref{ctp_is_admissible}).   We then prove:

\begin{theorem}[=\Cref{key_cassels_tate_prop_4_sel}] \label{key_cassels_tate_prop_4_sel_intro}
Suppose that $\textup{Sel}^2(A/K)\cap H^1(G,A[2])$ is generated by the class $\mathfrak{c}$.
Let $P$ be an admissible pairing on $\textup{Sel}^2(A/K)$. Then there is a quadratic character $\chi:G_K\rightarrow \mu_2$ such that all of the following hold:
\begin{itemize}
\item[(i)] $\textup{rk}_2(A^{\chi}/K)\equiv \textup{rk}_2(A/K) \pmod 2,$
\item[(ii)] $\textup{Sel}^2(A^{\chi}/K)=\textup{Sel}^2(A/K)$ as subgroups of $H^1(K,A[2])$,
\item[(iii)] for all $\mathfrak{a},\mathfrak{b}\in \textup{Sel}^2(A^{\chi}/K)$, we have 
\[\textup{CTP}_{\chi}(\mathfrak{a},\mathfrak{b})=P(\mathfrak{a},\mathfrak{b}).\]
\end{itemize}
Here $\textup{CTP}_\chi$ denotes the Cassels--Tate pairing on the $2$-Selmer group of $A^\chi$.
\end{theorem}

\begin{remark}
As above, $\mathfrak{c}$ is necessarily contained in $\textup{Sel}^2(A/K)\cap H^1(G,A[2])$ and if $G\cong \textup{Sp}_{2d}(\mathbb{F}_2)$ then $H^1(G,A[2])$ itself is generated by the class $\mathfrak{c}$.  Thus in this case the condition on $\textup{Sel}^2(A/K)\cap H^1(G,A[2])$ is automatically satisfied.
\end{remark}

\begin{remark}
\Cref{key_cassels_tate_prop_4_sel_intro} is perhaps not surprising in light of recent breakthrough work of Smith \cite{smith2022distribution1,smith2022distribution2} which, amongst other things, establishes strong randomness results for Cassels--Tate pairings in twist families (earlier work of Smith \cite{smi16} also contains a variant of \Cref{key_cassels_tate_prop_4_sel} applying to elliptic curves with full rational $2$-torsion). We remark however that the possible failure of the Cassels--Tate pairing to be alternating causes issues there also, as a result of which Smith's work currently applies only to settings where the class $\mathfrak{c}$ vanishes. The goal of understanding the behaviour of the Cassels--Tate pairing in twist families for which the class  $\mathfrak{c}$ is non-trivial was a key motivator for carrying out the current project. 
\end{remark}

%
%

\subsection{Notation and conventions}

For a field $F$,  we denote by $\overline{F}$ a fixed separable closure of $F$, and denote by $G_F$ the absolute Galois group of $F$. For an integer $n$ coprime to the characteristic of $F$, we denote by $\mu_n$ the $G_F$-module of $n$-th roots of unity in $\overline{F}$.
For a profinite group $\Gamma$,  a discrete $\Gamma$-module $M$, and $i\geq 0$, we denote by 
$C^i(\Gamma,M)$, $Z^i(\Gamma,M)$,  and $B^i(\Gamma,M)$
the groups of continuous $i$-cochains, $i$-cocycles and $i$-coboundaries respectively. We denote by 
\[d:C^i(\Gamma,M)\longrightarrow C^{i+1}(\Gamma,M)\]
the coboundary map, and denote by $H^i(\Gamma,M)=Z^i(\Gamma,M)/B^i(\Gamma,M)$ the $i$-th cohomology group of $\Gamma$ with coefficients in $M$. Given a field $F$ and a discrete $G_F$-module $M$, we write $H^i(F,M)$ for the corresponding Galois cohomology group, i.e. as shorthand for $H^i(G_F,M)$.

For a number field $K$,  and a place $v$ of $K$, we denote by $K_v$ the completion of $K$ at $v$. We denote by  $\textup{inv}_v$ the local invariant map, and when $v$ is nonarchimedean denote by $k_v$ the residue field of $K_v$.    For each place $v$ we fix an embedding $i_v:\overline{K}\hookrightarrow \overline{K_v}$. This induces an inclusion $G_{K_v}\hookrightarrow G_K$, and a corresponding restriction map $\textup{res}_v$ on cochains, cocycles and  Galois cohomology groups. Further, $i_v$ fixes an extension of $v$ to $\overline{K}$; in this way, for each finite Galois extension $F/K$ unramified at $v$, we have a well defined Frobenius element $\textup{Frob}_v$ in $\textup{Gal}(F/K)$.

For each nonarchimedean place $v$ of $K$, we denote by $K_v^{\textup{ur}}$ the maximal unramified extension of $K_v$. Given a $G_{K_v}$-module $M$, we denote by $H^1_{\textup{ur}}(K_v,M) \subseteq H^1(K_v,M)$ the subgroup of unramified   classes, so that  $H^1_{\textup{ur}}(K_v,M)$ is isomorphic  to  $H^1\big(\textup{Gal}(K^{\textup{ur}}_v/K_v),M^{I_v}\big)$. Here $I_v$ denotes the inertia subgroup of $G_{K_v}$. Evaluation of cocyles at the Frobenius element $\textup{Frob}_v\in \textup{Gal}(K^{\textup{ur}}_v/K_v)$ induces an isomorphism
\begin{equation} \label{ev_at_frob_unram_iso}
 H^1\big(\textup{Gal}(K^{\textup{ur}}_v/K_v),M^{I_v}\big)\stackrel{\sim}{\longrightarrow}M^{I_v}/(\textup{Frob}_v-1)M^{I_v}.
\end{equation}

\subsection{Acknowledgements}
We would like to thank Dami\'{a}n Gvirtz-Chen for helpful comments and for sharing  \textsc{magma}  code used to determine the subgroups of $S_6$ listed in \Cref{rem:dim_2_subgroups}. We would also like to thank Alexei Skorobogatov for encourgaing us to look at this problem and for pointing out helpful references. We thank Netan Dogra for a helpful discussion.

This work has been supported by the Engineering and Physical Sciences Research Council (EPSRC) grant EP/V006541/1 ‘Selmer groups, Arithmetic Statistics and Parity Conjectures'.

\section{Setup and preliminary results}
For the rest of the paper, we take $A$ to be an abelian variety defined over a number field $K$. Denote by $A^\vee$ the dual abelian variety, and suppose that $A$ is equipped with a fixed principal polarisation $\lambda:A\rightarrow A^\vee$ defined over $K$. 

\subsection{Pairings associated to $\lambda$} 
Suppressing the polarisation $\lambda$ from the notation, denote by 
\begin{equation} \label{weil_pairing_eq}
e_2:A[2]\times A[2]\longrightarrow \mu_2
\end{equation}
the Weil pairing on $A[2]$ associated to $\lambda$, so that $e_2$ is the composition of   the map $1\times \lambda$ with the usual Weil pairing  $A[2]\times A^\vee[2]\rightarrow \mu_2$. Similarly, writing $\Sha(A/K)$ for the Shafarevich--Tate group of $A$, we denote by 
\begin{equation} \label{ctp_pairing}
\textup{CTP}:\Sha(A/K)\times \Sha(A/K)\longrightarrow \mathbb{Q}/\mathbb{Z}
\end{equation}
the Cassels--Tate pairing on $\Sha(A/K)$ associated to $\lambda$, so that \eqref{ctp_pairing} is the pairing denoted $\left \langle~,~\right \rangle_\lambda$ in   \cite{MR1740984}. By an abuse of notation, we also denote by $\textup{CTP}$ the pairing 
\begin{equation}
\textup{Sel}^2(A/K)\times \textup{Sel}^2(A/K)\longrightarrow \tfrac{1}{2}\mathbb{Z}/\mathbb{Z}
\end{equation}
on the $2$-Selmer group of $A$, given by projecting to $\Sha(A/K)[2]$ and  applying the pairing  \eqref{ctp_pairing}.

\subsection{The class $\mathfrak{c}$} \label{ssec:class_c} 

The pairing \eqref{weil_pairing_eq} is non-degenerate, alternating, and $G_K$-equivariant. In particular, $G_K$ acts on $A[2]$ through the symplectic group $\textup{Sp}(A[2])$ of automorphisms  preserving the Weil pairing. In this way, we   identify $G=\textup{Gal}(K(A[2])/K)$ with a subgroup of $\textup{Sp}(A[2])$.

Let $\mathfrak{c}_\lambda \in H^1(K,A^\vee[2])$ be the torsor defined in \cite[Equation (17)]{MR2915483}, so that $\mathfrak{c}_\lambda$ parameterises quadratic refinements of the Weil pairing  on $A[2]$ (equivalently, symmetric line bundles on $A$ inducing the given principal polarisation). Write $\mathfrak{c}=\lambda^{-1}(\mathfrak{c}_\lambda)$ for the corresponding class in  $H^1(K,A[2])$. Per \cite[Remark 3.3]{MR2915483}, the class $\mathfrak{c}$ lies in $\textup{Sel}^2(A/K)$. Further, combining that remark with \cite[Theorem 5]{MR1740984}, we see that for all $\mathfrak{a}\in \Sha(A/K)$, we have
\begin{equation}\textup{CTP}(\mathfrak{a},\mathfrak{a})=\textup{CTP}(\mathfrak{a},\mathfrak{c}).
\end{equation}
Via the description of $\mathfrak{c}$ in terms of quadratic refinements of the Weil pairing, we see that $\mathfrak{c}$ lies in the image of inflation from $H^1(G,A[2])$. In what follows, we will often view $\mathfrak{c}$ as an element of $H^1(G,A[2])$ without further comment. The proof of \cite[Proposition 3.6 (a)]{MR2915483} shows that $\mathfrak{c}$ in fact lies in the group 
\[\bigcap_{g\in G}\ker\Big(H^1(G,A[2])\stackrel{\textup{res}}{\longrightarrow}  H^1(\left \langle g\right \rangle, A[2]) \Big) ,\]
where, for $g\in G$, we denote by $\left \langle g \right \rangle$ the subgroup  generated by $g$.

\subsection{Quadratic twists}

For $\chi:G_K\rightarrow \mu_2$ a (possibly trivial) quadratic character, we denote by $A^\chi/K$ the quadratic twist of $A$ by $\chi$. This is an abelian variety over $K$, equipped with a $\overline{K}$-isomorphism $\psi:A\stackrel{\sim}{\longrightarrow}A^\chi$ satisfying 
$\psi^{-1}\circ {}^\sigma \psi=\chi(\sigma)$
for al $\sigma \in G_K$. 
 Note that $\psi$ restricts to a $G_K$-module isomorphism
 \begin{equation}  \label{2_tors_ident}
 A[2]\stackrel{\sim}{\longrightarrow}A^\chi[2].
 \end{equation}
In what follows, we will always identify $A^\chi[2]$ with $A[2]$ via this map. 

Per \cite[Lemma 4.16]{MR3951582}, the polarisation $\lambda$ on $A$ descends via $\psi$ to a principal polarisation on $A^\chi$ defined over $K$.  In this way we view $A^\chi$ as a principally polarised abelian variety also. As explained in e.g. \cite[Section 5.5]{MR3951582}, the isomorphism \eqref{2_tors_ident} identifies the Weil pairings. In particular, it identifies $\mathfrak{c}$ with the torsor constructed from quadratic refinements of the Weil pairing on $A^\chi[2]$. Consequently, when we view $\textup{Sel}^2(A^\chi/K)$ inside $H^1(K,A[2])$ via \eqref{2_tors_ident}, as we do henceforth, we have $\mathfrak{c}\in \textup{Sel}^2(A^\chi/K)$ also.

\subsection{Tate local duality} \label{local_tate_pairing_identified_rmk} 
For each place $v$ of $K$, Tate local duality gives a non-degenerate pairing 
\begin{equation} \label{eq:local_tate}
H^1(K_v,A[2])\otimes H^1(K_v,A[2])\stackrel{\cup}{\longrightarrow}H^2(K_v,\mu_2)\hookrightarrow \mathbb{Q}/\mathbb{Z},
\end{equation}
where the first map is the cup-product associated to the Weil pairing and the second is the local invariant map. We refer to this as the local Tate pairing. We denote by $\mathscr{S}_v(A)$ the image of 
  the connecting map
\[\delta:A(K_v)/2A(K_v)\hookrightarrow H^1(K_v,A[2]),\]
 associated to the multiplication-by-$2$ Kummer sequence for $A$.  Per \cite[Proposition 4.10]{MR2833483}, for each place $v$ of $K$, the subgroup $\mathscr{S}_v(A)$ is a maximal isotropic subspace of $H^1(K_v,A[2])$ with respect to the local Tate pairing.

\subsection{$2$-Selmer conditions} \label{selmer_conds_intro_notat}

Given a place $v$ of $K$, and a quadratic character $\chi:G_{K_v}\rightarrow \mu_2$, we denote by $\mathscr{S}_v(A,\chi)$ the subspace of $H^1(K_v,A[2])$ corresponding, under \eqref{2_tors_ident}, to the Kummer image $\mathscr{S}_v(A^\chi)$. Since \eqref{2_tors_ident} identifies the Weil pairings on $A[2]$ and $A^{\chi}[2]$, for each place $v$  the induced isomorphism $H^1(K_v,A[2])\cong H^1(K_v,A^{\chi}[2])$ identifies the local Tate pairings. In particular, each $\mathscr{S}_v(A,\chi)$ is a maximal isotropic subspace of $H^1(K_v,A[2])$ with respect to \eqref{eq:local_tate} also.

 In the statement of the following lemma, $N_A$ denotes the conductor of $A$. Further, for a place $v$ of $K$, we denote by $\Phi(A/K_v)$ the group of $\overline{k}_v$-points of the component group of the N\'{e}ron model of $A$, viewed as a $\textup{Gal}(K^{\textup{ur}}_v/K_v)$-module.


%
%
%

\begin{lemma} \label{selmer_conds_lemma}
Let $v$ be a place of $K$, and let $\chi:G_{K_v}\rightarrow \mu_2$ be a (possibly trivial) quadratic character. Then for each place $v$ of $K$ not dividing $2\infty$, we have:
\begin{itemize}
\item[(i)]   $\dim H^1(K_v,A[2])=2\dim A(K_v)[2]$,
\item[(ii)] if  $\chi$ is unramified and  $\Phi(A/K_v)$ has odd order (in particular if $v\nmid N_A$), then 
\[\mathscr{S}_v(A,\chi)=H^1_{\textup{ur}}(K_v,A[2]),\]
\item[(iii)] if $v\nmid N_A$ and $\chi$ is ramified then 
\[\mathscr{S}_v(A,\chi)\cap H^1_{\textup{ur}}(K_v,A[2])=0.\]
\end{itemize}
\end{lemma}

\begin{proof}
(i). Since $v\nmid 2\infty$, we have 
\[\dim \mathscr{S}_v(A)=\dim A(K_v)/2A(K_v)=\dim A(K_v)[2]. \]
Since $\mathscr{S}_v(A)$ is  maximal isotropic for the local Tate pairing,  we have $\dim H^1(K_v,A[2])=2\dim \mathscr{S}_v(A).$

(ii).  Since N\'{e}ron models commute with unramified base change, $\Phi(A^\chi/K_v)$ has odd order also. Thus we can assume that $\chi$ is trivial. From \eqref{ev_at_frob_unram_iso}   we have
\[\dim H^1_{\textup{ur}}(K_v,A[2])=\dim A(K_v)[2]=\dim \mathscr{S}_v(A),\]
hence it suffices to show that $H^1_{\textup{ur}}(K_v,A[2])$ is contained in $\mathscr{S}_v(A)$. Equivalently, we want to show that the image of $H^1_{\textup{ur}}(K_v,A[2])$ in $H^1(K_v,A)$ is trivial. 
By  \cite[Proposition I.3.8]{MR2261462} we have \[H^1_{\textup{ur}}(K_{v},A)\cong H^1_{\textup{ur}}(K_{v},\Phi(A/K_v)).\]
Since $\Phi(A/K_v)$ has odd order, it follows that the group $H^1_{\textup{ur}}(K_{v},A)[2]$ is trivial. Thus the image of $H^1_{\textup{ur}}(K_v,A[2])$ in $H^1(K_v,A)$ is trivial also. 

(iii). Combine part (ii) with \cite[Lemma 4.3]{HS16}.
\end{proof}
%

 The following lemma, based on \cite[Corollary 3.4(ii)]{MazRub2010}, can be used to produce lots of twists with the same $2$-Selmer group. In the statement, $\textup{rk}_2(A/K)$ denotes the $2^\infty$-Selmer rank of $A$.

\begin{lemma} \label{selmer_agreement_lemma}
Let $\Sigma$ be a finite set of places of $K$ containing all places dividing $2N_A\infty$. Suppose that $\chi:G_K\rightarrow \mu_2$ is a quadratic character such that:
\begin{itemize}
\item $\textup{res}_v(\chi)$ is trivial  for all $v\in \Sigma$,
\item if   $\chi$ is ramified at a place $v \notin \Sigma$, then  $A[2]^{\textup{Frob}_v}=0$. 
\end{itemize}
Then   $\textup{Sel}^2(A^{\chi}/K)=\textup{Sel}^2(A/K)$ as subgroups of $H^1(K,A[2])$. Moreover, we have 
\[\textup{rk}_2(A^{\chi}/K)\equiv \textup{rk}_2(A/K) \pmod 2.\]
\end{lemma}

\begin{proof}
For the equality of $2$-Selmer groups we will show that $\mathscr{S}_v(A)=\mathscr{S}_v(A,\textup{res}_v(\chi))$ for all places $v$. For places $v\in \Sigma$ this holds since $\textup{res}_v(\chi)$ is trivial. For places $v\notin \Sigma$ where $\chi$ ramifies, we have $A(K_v)[2]=0$, hence $H^1(K_v,A[2])=0$ by \Cref{selmer_conds_lemma}(i).  At the remaining places the character $\chi$ is unramified, and we conclude from  \Cref{selmer_conds_lemma}(ii).

 The claim regarding the parity of $2^\infty$-Selmer ranks follows from  \cite[Theorem 10.20]{MR3951582}, which shows that the parity of $\textup{rk}_2(A^\chi/K)$ depends only on the restriction of $\chi$ to $G_{K_v}$ for places $v$ dividing $2N_A\infty$ (cf. also \Cref{parity_2_infinity} below).
\end{proof}

\section{The behaviour of the Cassels--Tate pairing under quadratic twist}
 
Let $\chi$ be a quadratic character. Viewing both $\textup{Sel}^2(A/K)$ and $\textup{Sel}^2(A^\chi/K)$ as subgroups of $H^1(K,A[2])$, the sum of the Cassels--Tate pairings for $A$ and $A^\chi$  gives a pairing on the intersection of $\textup{Sel}^2(A/K)$ and $ \textup{Sel}^2(A^{\chi}/K)$. As observed by several authors  \cite{smi16, Har19, MR3951582}, this pairing is significantly simpler than either of the individual Cassels--Tate pairings. We give a description of this pairing  in a special case, based  on \cite[Section 8.2]{MS21}. In that work, denoting by $F$ the quadratic extension corresponding to $\chi$, and denoting by $\textup{R}_{F/K}$ restriction of scalars from $F$ to $K$, the simplification is obtained by exploiting  a certain isogeny 
\[\textup{R}_{F/K}A\longrightarrow A\times A^\chi\] splitting multiplication by $2$.  

\subsection{The Cassels--Tate pairing on the $2$-Selmer group under quadratic twist}

As in \eqref{ctp_pairing}, denote by $\textup{CTP}$ the Cassels--Tate pairing on $\textup{Sel}^2(A/K)$. For a quadratic character $\chi:G_K\rightarrow \mu_2$,  
denote by $\textup{CTP}_\chi$ the corresponding pairing on $\textup{Sel}^2(A^\chi/K)$.  

\begin{defi} \label{psi_b_b_defi}
Let $S$ denote the subset of elements $\sigma \in G_K$  for which $A[2]^\sigma=0$. For any $\sigma \in S$ we have $A[2]/(\sigma-1)A[2]=0$, so given a $1$-cocycle $a:G_K\rightarrow A[2]$, there is a, necessarily unique, element $P_{a,\sigma}\in A[2]$ satisfying 
\[a(\sigma)=\sigma P_{a,\sigma}-P_{a,\sigma}.\]
Note that if $v\nmid 2N_A\infty$ is a place of $K$ such that $a$ is unramified at $v$, then $\textup{res}_v(a)=dP_{a,\textup{Frob}_v}$. 

Suppose we are also given a $1$-cocycle $b:G_K\rightarrow A[2]$ such that the class of $a\cup b=0$ is zero in $H^2(G_K,\mu_2)$, and let $\gamma_{a,b}:G_K\rightarrow \mu_2$ be a $1$-cochain such that $d\gamma_{a,b}=a\cup b$.  Identifying $\tfrac{1}{2}\mathbb{Z}/\mathbb{Z}$ with $\mu_2$, and writing both groups additively, define  $\psi_{a,b}:S\rightarrow \tfrac{1}{2}\mathbb{Z}/\mathbb{Z}$ to be the function sending $\sigma\in S$ to 
\[\psi_{a,b}(\sigma):=(P_{a,\sigma}\cup b)(\sigma)+\gamma_{a,b}(\sigma)=e_2\big(P_{a,\sigma},b(\sigma)\big)+\gamma_{a,b}(\sigma).\]
The importance of the function $\psi_{a,b}$ comes from its role in \Cref{main_ct_comp_prop} below.
\end{defi}

\begin{remark} \label{rem:local_chars}
The function $\psi_{a,b}$ is constant on conjugacy classes. Indeed, given $\sigma\in S$ and $\tau\in G_K$,  one computes
\[P_{a,\tau \sigma \tau^{-1}}=\tau P_{a,\sigma}+a(\tau).\]
From this it follows easily that
$\psi_{a,b}(\tau\sigma \tau^{-1})=\psi_{a,b}(\sigma).$ We remark also that if $v\nmid 2N_A\infty$   is a place of $K$ such that $A[2]^{\textup{Frob}_v}=0$ and such that $a$, $b$, and $\gamma_{a,b}$ are unramified at $v$, then 
\[\textup{res}_v\big(P_{a,\textup{Frob}_v}\cup b + \gamma_{a,b}\big)\] is an unramified quadratic character  of $G_{K_v}$. Indeed, it has trivial coboundary since $dP_{a,\textup{Frob}_v}=\textup{res}_v(a)$ and $d\gamma_{a,b}= a\cup b$. Up to identifying $\mu_2$ with $\tfrac{1}{2}\mathbb{Z}/\mathbb{Z}$, the quantity $\psi_{a,b}(\textup{Frob}_v)$ is the result of evaluating this local character at the Frobenius element in $G_{K_v}$. See also \Cref{magnus_remark} below for another interpretation of these local characters. 
\end{remark}

 The following is an analogue of \cite[Theorem 3.2]{smi16} and \cite[Proposition 3.29]{Har19}, which both prove a similar result in the case that $A[2]\subseteq A(K)$. In the statement, for a quadratic character $\chi:G_K\rightarrow \mu_2$, we denote by $\textup{Ram}(\chi)$ the set of nonarchimedean places of $K$ at which $\chi$ is ramified. 

\begin{proposition} \label{main_ct_comp_prop}
Let $\mathfrak{a}$ and $\mathfrak{b}$ be elements of $\textup{Sel}^2(A/K)$, and represent them by $1$-cocycles $a$ and $b$ respectively. Let $\gamma_{a,b}:G_K\rightarrow \mu_2$ be a $1$-cochain such that $d\gamma_{a,b}=a\cup b$, and let $\psi_{a,b}$ be defined from these quantities as in \Cref{psi_b_b_defi}. Let $\Sigma$ be a finite set of places of $K$ containing all places dividing $2N_A\infty$, and such that $\gamma_{a,b}$ is unramified outside $\Sigma$. 

Suppose that $\chi:G_K\rightarrow \mu_2$ is a quadratic character such that  $\textup{res}_v(\chi)$ is trivial for all $v\in \Sigma$, and such that $A[2]^{\textup{Frob}_v}=0$ for all places  $v\in \textup{Ram}(\chi)\setminus \Sigma$. Then $\mathfrak{a}$ and $\mathfrak{b}$ lie in  $\textup{Sel}^2(A^{\chi}/K)$,   and we have
\[\textup{CTP}(\mathfrak{a},\mathfrak{b})+\textup{CTP}_{\chi}(\mathfrak{a},\mathfrak{b})=\sum_{v\in \textup{Ram}(\chi)\setminus \Sigma}\psi_{a,b}(\textup{Frob}_v).\]
\end{proposition}
 
 \begin{remark}
That one can find a cochain $\gamma_{a,b}$ as in the statement  follows from reciprocity for the Brauer group of $K$ and the fact that, for all places $v$, the Kummer image $\mathscr{S}_v(A)$ is  isotropic for the local Tate pairing \eqref{eq:local_tate}. 
 \end{remark}
 
\begin{proof}[Proof of \Cref{main_ct_comp_prop}]
That $\mathfrak{a}$ and $\mathfrak{b}$ lie in   $\textup{Sel}^2(A^{\chi}/K)$ follows from \Cref{selmer_agreement_lemma}. 

It is shown in \cite[Section 8.2]{MS21} that $\textup{CTP}(\mathfrak{a},\mathfrak{b})+\textup{CTP}_\chi(\mathfrak{a},\mathfrak{b})$ is the result of pairing  $\mathfrak{a}$ and $\mathfrak{b}$ under the Cassels--Tate pairing $\textup{CTP}_E$ associated, via \cite[Definition 3.2]{MS222}, to the exact sequence displayed as \cite[Equation (8.4)]{MS21}. That is, to the exact sequence 
\begin{equation} \label{morgan_smith_seq}
E=\Big[0\rightarrow \big(A[2],\mathscr{W}_2+\mathscr{W}_2^{\chi}\big)\stackrel{i}{\longrightarrow}   \big(R_{F/K}A[2],\mathscr{W}_{2}'\big) \stackrel{\nu}{\longrightarrow} (A[2],\mathscr{W}_2\cap\mathscr{W}_2^{\chi})\rightarrow 0\Big],
\end{equation}
in the category $\textup{SMod}_K$  of \cite[Definition 1.1]{MS222}. Let us explain the terms in this sequence.  To describe the underlying exact sequence of $G_K$-modules, note that $\textup{R}_{F/K}A[2]$ is naturally identified with the $G_K$-module $\textup{Ind}_{F/K}A[2]= \mathbb{Z}[G_K]\otimes_{\mathbb{Z}[G_F]}A[2]$. With this identification, the maps $i$ and $\nu$ are defined by setting, for $x\in A[2]$ and $\sigma \in G_K$,
\[i(x) =\sum_{\sigma \in G_K/G_F}\sigma \otimes \sigma^{-1}x \quad \textup{ and }\quad \nu(\sigma \otimes x)= \sigma x.\]
In the notation of \Cref{local_tate_pairing_identified_rmk}, the local conditions subgroups are given by
\[\mathscr{W}_2=\prod_{v}\mathscr{S}_v(A),\quad \mathscr{W}_2^{\chi}=\prod_v\mathscr{S}_v(A,\textup{res}_v(\chi)),\quad \textup{ and }\quad \mathscr{W}_2'= \prod_v\mathscr{S}_v(R_{F/K}A).\]

We now follow \cite[Definition 3.2]{MS222} to compute $\textup{CTP}_E(\mathfrak{a},\mathfrak{b})$. 
With the representing cocycles $a$ and $b$ chosen as above, consider the $1$-cochain $f=1\otimes a$, valued in $\textup{Ind}_{F/K}A[2]$. By construction, $\nu(f)=a$. Further, one has $df=i(\chi \cup a)$. In the notation of \cite[Definition 3.2]{MS222}, we may thus take $\epsilon$ to be the $\mu_2$-valued $2$-cochain $\chi \cup \gamma_{a,b}$, so that 
  \[d\epsilon = \chi \cup a\cup b= df\cup b.\]
  
  The next step is to choose, for each place $v$ of $K$, a $1$-cocycle   $\overline{\phi}_v:G_{K_v}\rightarrow \textup{Ind}_{F/K}A$, such that $\nu(\overline{\phi}_v)=\textup{res}_v(a),$ and such that $\overline{\phi}_v$ represents a class in $\mathscr{S}_v(R_{F/K}A)$. Having done this, we have 
  \begin{equation} \label{ctp_ugly_formula}
  \textup{CTP}_E(\mathfrak{a},\mathfrak{b})=\sum_{v\textup{ place of }K}\textup{inv}_v\Big(i^{-1}(\textup{res}_v(f)-\overline{\phi}_v)\cup \textup{res}_v(b)- \textup{res}_v(\chi \cup\gamma_{a,b})\Big).\end{equation}
  
  For places $v\in \Sigma$, the assumption that  $\textup{res}_v(\chi)$ is trivial allows us to take $\overline{\phi}_v=1\otimes \textup{res}_v(a)$. Indeed, triviality of $\textup{res}_v(\chi)$ implies that  $R_{F/K}A$ is isomorphic to $A\times A$ over $K_v$, that the sequence of Galois-modules underlying  \eqref{morgan_smith_seq} splits over $K_v$, and that  $\mathscr{S}_v(R_{F/K}A)$ identifies with $\mathscr{S}_v(A)\oplus \mathscr{S}_v(A)$. With this choice of $\overline{\phi}_v$, we see that the contribution to \eqref{ctp_ugly_formula} from places $v\in \Sigma$ is trivial.    
  
Next, let $v$ be a place outside $\Sigma \cup \textup{Ram}(\chi)$. Then the $G_{K_v}$-module $\textup{Ind}_{F/K}A[2]$ is unramified at $v$, and $R_{F/K}A$ has good reduction at $v$. Thus $\mathscr{S}_v(R_{F/K}A)=H^1_{\textup{ur}}(K_v,\textup{Ind}_{F/K}A[2])$
and any permitted choice of $\overline{\phi}_v$ necessarily factors through $\textup{Gal}(K^{\textup{ur}}/K)$. Thus our choice of $\Sigma$ ensures that   the $\mu_2$-valued $2$-cocycle 
\[i^{-1}(\textup{res}_v(f)-\overline{\phi}_v)\cup \textup{res}_v(b)- \textup{res}_v(\chi \cup\gamma_{a,b})\]
factors through $\textup{Gal}(K^{\textup{ur}}/K_v)$. Since the group $H^2(\textup{Gal}(K^{\textup{ur}}/K_v),\mu_2)$ is trivial (cf.  \cite[Example in Chapter 3]{NSW08}), we see that the contribution to \eqref{ctp_ugly_formula} at $v$ is trivial also.  
  
Finally, let $v\in \textup{Ram}(\chi)\setminus \Sigma$, so that  $A[2]^{\textup{Frob}_v}=0$. To ease notation, write $P_v=P_{a,\textup{Frob}_v}$.  We may then take $\overline{\phi}_v=d(1\otimes P_v)$. Indeed, since $\overline{\phi}_v$ is a $1$-coboundary, its class in $H^1(K_v,R_{F/K}A[2])$ necessarily lies in   $\mathscr{S}_v(R_{F/K}A)$. Further, since $\nu$ is $G_{K_v}$-equivariant,  we have 
  \[\nu(\overline{\phi}_v)=d\circ \nu(1\otimes P_v)=dP_v=\textup{res}_v(a).\]
  The difference $\textup{res}_v(f)-\overline{\phi}_v$ is then equal to $i(\textup{res}_v(\chi) \cup P_v)$, so the contribution to \eqref{ctp_ugly_formula} at $v$ is 
  \[\textup{inv}_v\big(\chi \cup   (P_v\cup b+\gamma_{a,b})\big)= \psi_{a,b}(\textup{Frob}_v),\]   
  where we have dropped each instance of $\textup{res}_v$ from the left hand side to ease notation.
Here for the displayed equality we are using that $\chi$ is ramified at $v$, that $v\nmid 2$, and that  $P_v\cup b+\gamma_{a,b}$ is an unramified quadratic character of $G_{K_v}$.
  
  Summing over all places $v$ of $K$ gives the result. 
\end{proof}

\begin{remark}
Rather than drawing on the theory of \cite{MS222}, \Cref{main_ct_comp_prop} can also be proven from  \cite[Lemma 5.8]{MR3951582} via a direct cocycle computation. 
\end{remark}

 \begin{remark} \label{magnus_remark} 
Take  $a$, $b$ and $\gamma_{a,b}$ as in the statement of \Cref{main_ct_comp_prop}, and let $v\nmid 2N_A\infty$ be a place of $K$ such that $A[2]^{\textup{Frob}_v}=0$ and such that $a$, $b$, and $\gamma_{a,b}$ are unramified at $v$. The local quadratic characters $\textup{res}_v\big(P_{a,\textup{Frob}_v}\cup b + \gamma_{a,b}\big)$ discussed in \Cref{rem:local_chars}, and appearing in the proof of  \Cref{main_ct_comp_prop},  have the following interpretation. Let $V_a$ be the $G_K$-module with underlying abelian group $\mu_2\oplus A[2]$, but with $G_K$-action twisted by the cocycle $a$. Thus for $\sigma \in G_K$ and  $(\lambda,x)\in V_a$, we have
\[ \sigma  \cdot (\lambda, x)= \big(\sigma \lambda+e_2(a(\sigma),x)~,~ \sigma x\big).\]
Here again we are writing $\mu_2$ additively. The inclusion/projection maps on the level of abelian groups realise $V_a$ as an extension of $G_K$-modules
\begin{equation}\label{seq_magnus}
0\longrightarrow \mu_2\longrightarrow V_a \longrightarrow A[2]\longrightarrow 0.
\end{equation}
 Denoting by $\mathfrak{a}$ and $\mathfrak{b}$ the classes of $a$ and $b$ in $H^1(K,A[2])$, the resulting coboundary map $H^1(K,A[2])\rightarrow H^2(K,\mu_2)$ sends $\mathfrak{b}$ to $\mathfrak{a}\cup \mathfrak{b}$. In particular,  since $\mathfrak{a}\cup \mathfrak{b}=0$   we can lift $\mathfrak{b}$ to $H^1(K,V_a)$. One checks that such a lift is given explicitly by the class $\widetilde{\mathfrak{b}}$ of the $1$-cocycle 
\[\sigma \longmapsto (\gamma_{a,b}(\sigma), b(\sigma)).\]
The vanishing of both $H^0(K_v,A[2])$ and $H^1(K_v,A[2])$ imply that there is a unique $G_{K_v}$-equivariant splitting $s_v:V_a\rightarrow \mu_2$ of the sequence \eqref{seq_magnus} (such splittings are in one-to-one correspondence with elements $x$ of $A[2]$ satisfying $dx=\textup{res}_v(a)$). 
In particular, $s_v$ induces a map $H^1(K_v,V_a)\rightarrow H^1(K_v,\mu_2)$, from which we obtain an unramified quadratic character $\chi_v:G_{K_v}\rightarrow \mu_2$ as $\chi_v=s_v\circ \textup{res}_v(\widetilde{\mathfrak{b}})$. One checks that the map $s_v$ is given explicitly by sending $(\lambda,x)\in V_a$ to 
$ \lambda+e_2(P_{a,\textup{Frob}_v},x).$ From this it follows that
$\chi_v=\textup{res}_v\big(P_{a,\textup{Frob}_v}\cup b + \gamma_{a,b}\big).$

For an interpretation of the sequence \eqref{seq_magnus} in terms of the Kummer variety associated to $\mathfrak{a}$, see \cite[Section 3.2]{Har19}. 
 \end{remark}

\section{Torsors under $A[2]$} \label{big_torsors_section} 

\Cref{main_ct_comp_prop} will ultimately allow us to use the Chebotarev density theorem to control the behaviour of the Cassels--Tate pairing under quadratic twist. However, in order to do this, it will be necessary to understand the extensions of $K$ cut out by the cochains $\gamma_{a,b}$ appearing in the statement. We turn to this now. Write $L=K(A[2])$ and $G=\textup{Gal}(L/K)$. For $\sigma \in G_K$, denote by $\overline{\sigma}$ its image  in $G$. In this section, we assume that the abelian variety $A$ satisfies part  (A) of \Cref{existence_of_frob_elements_prop_1}. That is, that $A[2]$ is a simple $\mathbb{F}_2[G]$-module, and $\textup{End}_G(A[2])=\mathbb{F}_2$. 

\subsection{Field extensions defined by cocycles} \label{ssec:cocyclefieldext}

\begin{notation}\label{KT_extension}
Let $\mathscr{S}$ be a finite subgroup of $H^1(K,A[2])$. Let $T=\{\mathfrak{a}_1,...,\mathfrak{a}_{n}\}$ be a $\mathbb{F}_2$-linearly independent subset of $\mathscr{S}$ whose elements project to a basis for 
\[\mathscr{S}/\mathscr{S}\cap H^1(G,A[2])\stackrel{\textup{res}_L}{\hooklongrightarrow}H^1(L,A[2]).\]
Let $a_1,...,a_{n}$   be $1$-cocycles representing the classes of $\mathfrak{a}_1,...,\mathfrak{a}_n$ and define the homomorphism
\[\varphi_{T}: G_K\longrightarrow A[2]^n \rtimes G\]
by setting
\[\varphi_T(\sigma)=\big((a_i(\sigma))_{i=1}^n, \overline{\sigma}\big).\]
Here the action of $G$ on $A[2]^n$ is the diagonal one. Define $K_T$ to be the fixed field of $\ker(\varphi_T)$, so that 
 the map $\varphi_T$  induces an injection 
\begin{equation} \label{eq:varphi_semidirect}
\varphi_T: \textup{Gal}(K_T/K) \hookrightarrow A[2]^n \rtimes G.
\end{equation}
\end{notation}

We remark that the following proposition crucially uses part  (A) of \Cref{existence_of_frob_elements_prop_1}. 

\begin{proposition} \label{iso_to_semi_direct_general}
The map \eqref{eq:varphi_semidirect} is an isomorphism. 
\end{proposition}

\begin{proof}
The argument for (i) implies (iv) in  \cite[Proposition 3.2]{HS16} works in this setting. 
\end{proof}

 \subsection{Extension classes} \label{ssec:ext_classes}

\begin{notation}
With $n$ as above, write $\Gamma=A[2]^n \rtimes G$. 
 Write $[n]=\{1,...,n\}$. For each $i\in [n]$, let $f_i:\Gamma \rightarrow A[2]$ denote  the projection onto the $i$-th factor of $A[2]^n$.
This is a $1$-cocycle  valued in $A[2]$, and we write $\mathfrak{f}_i$ for its class in $H^1(\Gamma_n,A[2])$. With $\varphi_T$ as in   \eqref{eq:varphi_semidirect}, we have $a_i=f_i \circ \varphi_T$. 

Let $T'=\{\mathfrak{a}_{n+1},...,\mathfrak{a}_{n+m}\}$ be a basis for 
\[\mathscr{S}\cap H^1(G,A[2])=\ker\big(\mathscr{S}\stackrel{\textup{res}_L}{\longrightarrow}H^1(L,A[2])\big).\]
We allow for the possibility that $T'$ is empty. Note that $T\sqcup T'$ is a basis for $\mathscr{S}$.
For each $i\in [n+m]\setminus [n]$, denote by  $\mathfrak{f}_i$ the image of $\mathfrak{a}_i$ in $H^1(\Gamma,A[2])$ under the inflation map $H^1(G,A[2])\rightarrow H^1(\Gamma,A[2])$. 

Finally, define $\mathcal{S}=\mathcal{S}_0\sqcup \mathcal{S}_1$, where 
\[\mathcal{S}_0=\big( [n+m]\setminus [n]\big)\times [n]\quad \textup{ and }\quad\mathcal{S}_1=\{(i,j)\in [n]^2~~\colon~~ i<j \}.\]
\end{notation}

  Composition of cup-product and the Weil pairing gives a pairing
 \[\cup :H^1(\Gamma,A[2])\times H^1(\Gamma,A[2])\longrightarrow H^2(\Gamma,\mu_2).\]
For each $(i,j)\in [n+m]^2$ we obtain via this pairing a class $\mathfrak{f}_i\cup \mathfrak{f}_j$ in $H^2(\Gamma,\mu_2)$. 

 
 \begin{proposition} \label{independent_cups_prop}  
The $mn+n(n-1)/2$ classes 
\[\big(\mathfrak{f}_i\cup \mathfrak{f}_j\big)_{(i,j)\in \mathcal{S}} \]
  are linearly independent in $H^2(\Gamma,\mu_2)$.
 \end{proposition}
 
 \begin{proof}
 We first claim that the classes
$\mathfrak{f}_i\cup \mathfrak{f}_j$, for $(i,j)\in \mathcal{S}_1$,  
are linearly independent after restriction to $H^2(A[2]^n,\mu_2)$.   

In general, given a class $\boldsymbol \eta$ in $H^2(A[2]^n,\mu_2)$, there corresponds a central extension
 \[1\longrightarrow \mu_2 \longrightarrow E_{\boldsymbol \eta}\longrightarrow A[2]^n\longrightarrow 1.\]
To this we can associate a quadratic form $q_{\boldsymbol \eta}:A[2]^n\rightarrow \mu_2$, given by lifting to $E_{\boldsymbol \eta}$ and squaring.  Explicitly, representing $\boldsymbol \eta$ by a normalised $2$-cocycle $\eta$, for $v$ in $A[2]^n$ we have  
 \begin{equation*} \label{quad_form}
 q_{\boldsymbol\eta}(v)=\eta(v,v).
 \end{equation*}
As explained in e.g. \cite[Theorem 2.4.1]{Lur01}, the association sending a class $\boldsymbol \eta$ to $q_{\boldsymbol \eta}$ gives an isomorphism from $H^2(A[2]^n,\mu_2)$ to the $\mathbb{F}_2$-vector space of $\mu_2$-valued quadratic forms on $A[2]^n$. Given $i,j\in [n]$, one computes that the quadratic form $q_{ij}$ associated to $\mathfrak{f}_i\cup \mathfrak{f}_j$ sends an element $\textbf{v}=(v_k)_{k\in [n]}\in A[2]^n$ to
\[q_{ij}\big(\textbf{v}\big)=e_2(v_i,v_j).\]
Since these quadratic forms are visibly linearly independent as $i$ and $j$ range over elements of $[n]$ with $i<j$, the claim follows. 

Next, by \cite[Theorem 2]{tah72}, we have an exact sequence 
 \begin{equation}
 0 \longrightarrow H^1\big(G,\textup{Hom}(A[2]^n,\mu_2)\big)\stackrel{\alpha}{\longrightarrow} \widetilde{H}^2(\Gamma,\mu_2)\stackrel{\textup{res}}{\longrightarrow}H^2(A[2]^n,\mu_2),
 \end{equation}
 where $\widetilde{H}^2(\Gamma,\mu_2)$ is the subgroup of $H^2(\Gamma,\mu_2)$ consisting of classes vanishing on restriction to  $0\times G$,  $\textup{res}$ denotes restriction to  $A[2]^n$, and  $\alpha$ is defined in the proof of the cited result. Using the $n$-fold sum of the Weil pairing to identify $A[2]^n$ with  $\textup{Hom}(A[2]^n,\mu_2)$, one checks that  $\alpha$ sends a class $(\mathfrak{t}_i)_{i\in [n]}\in H^1(G,A[2])^n$ to the element 
 \[\sum_{i\in [n]}\mathfrak{t}_i\cup \mathfrak{f}_i.\]
 
 Now suppose we have a linear dependence relation 
 \[\sum_{(i,j)\in \mathcal{S}}\lambda_{ij}\cdot \mathfrak{f}_i\cup \mathfrak{f}_j =0\]
 in $H^2(\Gamma,\mu_2)$, for scalars $\lambda_{ij} \in \mathbb{F}_2$. Restricting to $A[2]^n$ and using the claim we conclude that $\lambda_{ij}=0$ for all $(i,j)\in \mathcal{S}_1$. From this we deduce that 
\[0=\sum_{j\in [n]}\Big(\textstyle{\sum_{i\in [n+m]\setminus [n]}}\lambda_{ij}\mathfrak{f}_i\Big)\cup \mathfrak{f}_j=\alpha\Big(\big(\textstyle{\sum_{i\in [n+m]\setminus [n]} }\lambda_{ij}\mathfrak{a}_i\big)_{j\in [n]}\Big).\]
The result now follows from injectivity of $\alpha$, and the linear independence of the $\mathfrak{a}_i$ as $i$ ranges over elements of $[n+m]\setminus [n]$.
 \end{proof}
 
The importance of the group  $E$ we now introduce comes from \Cref{determinantioN_of_cup_trivi_ext} below. 
 
 \begin{notation} \label{E_n_group} 
 For each $i\in [n+m]\setminus [n]$, fix a $1$-cocycle $a_i$ representing the class of $\mathfrak{a}_i$. Note that since every element of $A[2]$ is defined over $L$, the cocycle $a_i$ necessarily factors through $G$. For each such $i$ we denote by $f_i:\Gamma\rightarrow \mu_2$ the composition $\Gamma\rightarrow G \stackrel{f_i}{\rightarrow}A[2]$, the first map being the natural projection (thus $f_i$ is a $1$-cocycle representing the class of $\mathfrak{f}_i$). With this notation set, define the $2$-cocycle $\eta:\Gamma \rightarrow \mu_2^{\mathcal{S}}$ by 
\[\eta= (f_i\cup f_j)_{(i,j)\in \mathcal{S}} .\]
Denote by $E$ the corresponding central extension of $\Gamma$ by $\mu_2^{\mathcal{S}}$. Explicitly, as a set we have $E=\mu_2^{\mathcal{S}}\times \Gamma$, with the group structure given by 
 \[(\lambda, g)\cdot (\lambda',g')=(\lambda \lambda'\eta(g,g'),gg').\]
 The inclusion/projection maps on the level of sets fit into a short exact sequence 
 \begin{equation} \label{E_nextensionsequence}
 1\longrightarrow \mu_2^{\mathcal{S}}\longrightarrow E\longrightarrow \Gamma\longrightarrow 1,
 \end{equation}
 giving the extension structure on $E$.
 \end{notation}
 
 \begin{lemma} \label{trivial_chi_hom_en}
 Let $\psi:E\rightarrow \mu_2$ be a homomorphism. Then $\psi$ factors through $G$. 
 \end{lemma}
 
 \begin{proof}
Denote by $E'$ the central extension of $\Gamma$ by $\mu_2$ arising as the pushout of the sequence \eqref{E_nextensionsequence} by the restriction of $\psi$ to $\mu_2^{\mathcal{S}}$.
Since $\psi$ is defined on the whole of $E$, the inclusion of $\mu_2$ into $E'$ admits a section. Thus $E'$ is a split extension. On the other hand, the extension class of $E'$ corresponds to the image of (the class of) $\eta$ under the homomorphism 
\[H^2(\Gamma,\mu_2^{\mathcal{S}})\rightarrow H^2(\Gamma,\mu_2)\]
induced by $\psi$. If $\psi$ were non-trivial on restriction to $\mu_2^{\mathcal{S}}$, this image would be a non-trivial linear combination of the classes $\mathfrak{f}_i \cup \mathfrak{f}_j$, for $(i,j)\in \mathcal{S}$, hence non-trivial by  \Cref{independent_cups_prop}, a contradication. Thus $\psi$ has trivial restriction to $\mu_2^{\mathcal{S}}$, hence factors through $\Gamma$. Since moreover the proof of  \cite[Corollary 3.5]{HS16} shows that the abelianisation of $\Gamma$ coincides with the abelianisation of $G$, we conclude that $\psi$  factors through $G$.
 \end{proof}
 
 \subsection{Field extensions defined by trivialising cochains} \label{triv_cochain_extension_sec}
 
We henceforth make the following assumption:

\begin{assumption} \label{assumption_tirvial_cups}
Assume that,  for each $(i,j)\in \mathcal{S}$, we have 
$\mathfrak{a}_i\cup \mathfrak{a}_j=0$ in  $H^2(K,\mu_2).$
\end{assumption}
 
 \begin{notation} \label{trivial_cups_notat_choice}
For  $s=(i,j)\in \mathcal{S}$, fix  a $1$-cochain $\gamma_{s}:G_K\rightarrow \mu_2$ such that $ d\gamma_{s}=a_i\cup a_j$.   Such cochains exist by \Cref{assumption_tirvial_cups}. Recalling that
$a_i=f_i\circ \varphi_T$ for each $i\in [n]$, we see that the map 
$\varphi_{T,T'}:G_K\rightarrow E$
 defined by 
 \[\varphi_{T,T'}(\sigma)=  \Big(\big(\gamma_{s}(\sigma)\big)_{s\in \mathcal{S}}~,~\varphi_T(\sigma)\Big) \]
 is a homomorphism. 
Denote by $K_{T,T'}$ the fixed field of $\ker(\varphi_{T,T'})$, so that we have an induced injection  
 \begin{equation} \label{equation_E_n_extension}
 \varphi_{T,T'}:\textup{Gal}(K_{T,T'}/K)\hookrightarrow E.
 \end{equation}
 \end{notation}

 \begin{proposition} \label{determinantioN_of_cup_trivi_ext}
 The homomorphism \eqref{equation_E_n_extension} is an isomorphism. 
 Moreover, any quadratic subextension of $K_{T,T'}/K$ is contained in $L=K(A[2])$.
 \end{proposition}
 
 \begin{proof}  
 Since each cocycle $a_i$ is trivial on restriction to $\textup{Gal}(K_{T,T'}/K_T)$, the restriction of any $\gamma_{s}$ to $\textup{Gal}(K_{T,T'}/K_T)$ is a quadratic character 
 \[\chi_{s}:\textup{Gal}(K_{T,T'}/K_T) \longrightarrow \mu_2.\]
 We have a commutative diagram 
 \begin{equation}   
\begin{tikzcd}
 &  & \textup{Gal}(K_{T,T'}/K)  \arrow[d,"\varphi_{T,T'}"]\arrow[dr, "\varphi_T"]   \\
 1\ar[r]&\mu_2^{\mathcal{S}} \ar[r]& E\ar[r] &\Gamma \ar[r]& 1. 
\end{tikzcd}
\end{equation}
Since $\varphi_T$ is surjective (see \Cref{iso_to_semi_direct_general}), to show that \eqref{equation_E_n_extension} is an isomorphism it suffices to show that the map $\textup{Gal}(K_{T,T'}/K_T)\rightarrow \mu_2^{\mathcal{S}}$ induced by $\varphi_{T,T'}$ is surjective. For this, it suffices to show that the quadratic characters $\chi_{s}$, for $s\in \mathcal{S}$,
 are linearly independent. Suppose otherwise. Then we can find a non-trivial linear combination $\gamma$ of the corresponding cochains $\gamma_{s}$ which vanishes identically on $\textup{Gal}(K_{T,T'}/K_T)$. Using this and the fact that $d\gamma_{s}=a_i\cup a_j$ (where $i$ and $j$ are such that $s=(i,j)$) we deduce that  $\gamma$ factors through $\textup{Gal}(K_T/K)$. Via \Cref{iso_to_semi_direct_general} we can thus view $\gamma$ as a $\mu_2$-valued function on $\Gamma$, and computing the coboundary of $\gamma$ we  deduce that a non-trivial linear combination of the classes $\mathfrak{f}_i\cup \mathfrak{f}_j$, for $(i,j)\in \mathcal{S}$,  vanishes in $H^2(\Gamma_n,\mu_2)$. But this contradicts \Cref{independent_cups_prop}. 
 The claim concerning quadratic subextensions of $K_{T,T'}/K$  now follows from \Cref{trivial_chi_hom_en}.
 \end{proof}
 
 \begin{remark}
In the proof of \Cref{determinantioN_of_cup_trivi_ext}, the key step was to show that the quadratic characters $\{\textup{res}_{K_T}(\gamma_s)~~\colon~~s\in \mathcal{S}\}$ were linearly independent. An alternative approach to this is as follows. By \Cref{assumption_tirvial_cups}, each $a_i\cup a_j$ gives rise to a class in 
\[\ker\Big(H^2(\textup{Gal}(K_T/K),\mu_2)\stackrel{\textup{inf}}{\longrightarrow} H^2(K,\mu_2)\Big)\cong H^1(K_T,\mu_2)^{\textup{Gal}(K_T/K)}/\textup{res}_{K_T}H^1(K,\mu_2),\]
where the isomorphism is provided by the inflation-restriction exact sequence. One can check that, for $s=(i,j)\in \mathcal{S}$, the isomorphism from left to right sends the class of $a_i\cup a_j$ to $\textup{res}_{K_T}(\gamma_s)$. The sought linear  independence is then a direct consequence of \Cref{independent_cups_prop}, which shows that the classes of the $a_i\cup a_j$, for $(i,j)\in \mathcal{S}$, are linearly independent in $H^2(\textup{Gal}(K_T/K),\mu_2)$.
 \end{remark}
 
\Cref{determinantioN_of_cup_trivi_ext} shows that the extension of $K$ cut out by the cochains $\gamma_{s}$ is in some sense as large as possible. It is this result that, through \Cref{existence_of_frob_elements_prop} below, will allow us to combine \Cref{main_ct_comp_prop} with the Chebotarev density  to prove \Cref{key_cassels_tate_prop_4_sel_intro}. 
  
 \begin{notation} \label{notat_phiuhhk} 
Following \Cref{psi_b_b_defi}, denote by $S$ the subset of elements $\sigma$ in $\textup{Gal}(K_{T,T'}/K)$  for which $A[2]^\sigma=0$. For each $s=(i,j)\in \mathcal{S}$,   define  $\psi_{s}:S\rightarrow \frac{1}{2}\mathbb{Z}/\mathbb{Z}$  to be the function which, after composing with the unique isomorphism $\tfrac{1}{2}\mathbb{Z}/\mathbb{Z}\cong \mu_2$, sends $\sigma\in S$ to 
\begin{equation} \label{def_of_psi_2}
(P_{i,\sigma}\cup a_j)(\sigma) +\gamma_{s}(\sigma).
\end{equation}
Here $P_{i,\sigma}$  denotes the unique element of $A[2]$ satisfying $a_i(\sigma)=\sigma P_{i,\sigma}-P_{i,\sigma}$. As previously, we are writing $\mu_2$ additively.
\end{notation}
 


\begin{cor} \label{existence_of_frob_elements_prop}
For each $s\in \mathcal{S}$, fix $\epsilon_{s} \in \{0, \tfrac{1}{2}\}$. Suppose that $g \in G$ is such that $A[2]^g=0$.    Then there is $\sigma \in \textup{Gal}(K_{T,T'}/K)$, whose image in $G$ is $g$, and such that
 \begin{equation}
 \quad \psi_{s}(\sigma)=\epsilon_{s}\quad \textup{ for all }s\in \mathcal{S}.
 \end{equation}
\end{cor}

\begin{proof}
In the expression \eqref{def_of_psi_2}, the quantity 
$(P_{i,\sigma}\cup a_j)(\sigma)$
depends only on the image of $\sigma$ in $\textup{Gal}(K_T/K)$. However, it follows from \Cref{determinantioN_of_cup_trivi_ext} that,  after fixing the image of $\sigma \in \textup{Gal}(K_{T,T}/K)$ in $\textup{Gal}(K_T/K)$, we can still take the tuple $(\gamma_{s}(\sigma))_{s\in \mathcal{S}}$ to be any element of $\mu_2^{\mathcal{S}}$  we desire. 
%
%
%
\end{proof}

\section{Twisting to control the $4$-Selmer group} \label{twisting_4_sel}

We now turn to the proof of \Cref{key_cassels_tate_prop_4_sel_intro}. The final ingredient is a general result that guarantees the existence of global quadratic characters with specified local behaviour.  It is a variant of a result of Klagsbrun--Mazur--Rubin \cite[Proposition 6.8]{MR3043582} (see also \cite[Lemma 8.6]{MR3951582}).

\subsection{Quadratic characters with specified local behaviour}

For this subsection, fix a finite Galois extension $F/K$ with Galois group $\Gamma$. Fix also a finite set $\Sigma$ of places of $K$ containing all places dividing $2\infty$, and all places ramified in $F/K$. 

\begin{notation}
Let $\mathcal{C}\subseteq \Gamma$ be a non-empty union of conjugacy classes and write $\mathcal{P}(\mathcal{C})$ for the set of primes $\mathfrak{p}\notin \Sigma$ for which the Frobenius element $\textup{Frob}_\mathfrak{p}\in \Gamma$ lies in $\mathcal{C}$. Let $\mathcal{A}(\mathcal{C})$ denote the set of quadratic characters $\chi:\Gamma\rightarrow \{\pm 1\}$ satisfying $\chi(\sigma)=1$ for all $\sigma \in \mathcal{C}$. 
\end{notation}

\begin{proposition} \label{approximating_extensions_lemma}
Let  $(\chi_v)_{v\in \Sigma}$ be a collection of quadratic characters $\chi_v:G_{K_v}\rightarrow \mu_2$. Assume that, for every $\psi\in \mathcal{A}(\mathcal{C})$, we have
\begin{equation} \label{reciprocity_condition}
\sum_{v\in \Sigma}\textup{inv}_v(\chi_v\cup \textup{res}_v(\psi))=0.
\end{equation}
Then there is a global quadratic character $\chi:G_K\rightarrow \mu_2$ such that $\textup{res}_v(\chi)=\chi_v$ for all $v\in \Sigma$, and such that $\chi$ is unramified outside $\Sigma \cup \mathcal{P}(\mathcal{C})$.
\end{proposition}

\begin{proof}
Exactness at the middle term of the Poitou--Tate exact sequence (see, for example, \cite[Theorem I.4.10]{MR2261462}) applied to the set $\Sigma \cup  \mathcal{P}(\mathcal{C})$ of places, and to the self-dual $G_K$-module $\mu_2$,  shows that 
\[\text{im}\Big(H^1(K_{\Sigma \cup  \mathcal{P}(\mathcal{C})}/K,\mu_2) \stackrel{\textup{res}}{\longrightarrow}\sideset{}{'}\prod_{v\in \Sigma \cup  \mathcal{P}(\mathcal{C})} H^1(K_v,\mu_2)\Big)\]  is its own  orthogonal complement under the sum of the local Tate pairings. Here $K_{\Sigma \cup  \mathcal{P}(\mathcal{C})}$ denotes the maximal extension of $K$ unramified outside $\Sigma \cup \mathcal{P}(\mathcal{C})$, and the restricted direct product is taken with respect to the subgroups of  unramified classes. Projecting onto $\prod_{v\in \Sigma} H^1(K_v,\mu_2)$, it follows formally that the image of $H^1(K_{\Sigma \cup \mathcal{P}(\mathcal{C})}/K,\mu_2)$  in $\prod_{v\in \Sigma} H^1(K_v,\mu_2)$ is the orthogonal complement  of the image of
\begin{equation} \label{poitou_tate_kernel}
\ker\Big(H^1(K_{\Sigma }/K,  \mu_2)\stackrel{\textup{res}}{\longrightarrow }\sideset{}{'}\prod_{v\in \mathcal{P}(\mathcal{C})}H^1(K_v, \mu_2)\Big)
\end{equation}
 in $\prod_{v\in \Sigma}H^1(K_v,  \mu_2)$. Now let $\psi \in H^1(K_\Sigma/K,\mu_2)$. If either $\psi$ does not factor through $\Gamma$, of if there is $\sigma \in \mathcal{C}$ with $\psi(\sigma)=-1$, then by   the Chebotarev density theorem   we can find a prime $\mathfrak{p}\in \mathcal{P}(\mathcal{C})$ such that the restriction of $\chi$ to $G_{K_\mathfrak{p}}$ is non-trivial. From this we conclude that the kernel in  \eqref{poitou_tate_kernel} is equal to $\mathcal{A}(\mathcal{C})$, giving the result. 
\end{proof}

\subsection{The proof of \Cref{key_cassels_tate_prop_4_sel_intro}}

For the rest of this section we assume that $A$ satisfies parts $(A)$ and $(B)$ of \Cref{existence_of_frob_elements_prop_1}.
The following definition axiomatises the structure that the Cassels--Tate pairing on $\textup{Sel}^2(A/K)$ has; see \Cref{ctp_is_admissible} below. Recall that the class $\mathfrak{c}$ of \Cref{ssec:class_c}  necessarily lies in $\textup{Sel}^2(A/K)$. 

\begin{defi}
We say that a bilinear pairing \[P:\textup{Sel}^2(A/K)\times \textup{Sel}^2(A/K)\longrightarrow \tfrac{1}{2}\mathbb{Z}/\mathbb{Z}\] is \textit{admissible} if:
\[2P(\mathfrak{c},\mathfrak{c})\equiv  \dim \textup{Sel}^2(A/K)+\textup{rk}_2(A/K)\quad (\textup{mod } 2\mathbb{Z}),\]
and  
\begin{equation} \label{eq:PS_relation}
\phantom{hihihihihi}\quad \quad\quad \quad P(\mathfrak{a},\mathfrak{a})=P(\mathfrak{a},\mathfrak{c}) \quad \quad \textup{ for all }\mathfrak{a}\in \textup{Sel}^2(A/K).
\end{equation}
 This final relation implies, in particular, that $P$ is antisymmetric.
%
\end{defi}

%
%
%

\begin{lemma} \label{ctp_is_admissible}
Let $\chi:G_K\rightarrow \mu_2$ be a quadratic character such that \[\textup{rk}_2(A^\chi/K)\equiv \textup{rk}_2(A/K) \pmod 2,\] and such that, as subgroups of $H^1(K,A[2])$, we have $\textup{Sel}^2(A^\chi/K)=\textup{Sel}^2(A/K)$. Then $\textup{CTP}_\chi$ is an admissible pairing on $\textup{Sel}^2(A^\chi/K)$.
\end{lemma}

\begin{proof}
  That \eqref{eq:PS_relation} holds for $\textup{CTP}_\chi$ follows from \cite[Theorem 5]{MR1740984}. Further, since $A(K)[2]=0$,  since $\textup{rk}_2(A^\chi/K)$ is assumed to have the same parity as $\textup{rk}_2(A/K)$, and since $\dim \textup{Sel}^2(A/K)=\dim \textup{Sel}^2(A^\chi/K)$, we deduce that 
\[\dim \Sha_{\textup{nd}}(A^\chi/K)[2]\equiv \dim\textup{Sel}^2(A/K) + \textup{rk}_2(A/K) \pmod 2,\]
where $ \Sha_{\textup{nd}}(A^\chi/K)$ denotes the quotient of $\Sha(A^\chi/K)$ by its maximal divisible subgroup.
Combining this with \cite[Theorem 8]{MR1740984}  we see that
\[\textup{CTP}_{\chi}(\mathfrak{c},\mathfrak{c})=\frac{\dim\textup{Sel}^2(A^\chi/K) + \textup{rk}_2(A/K)}{2}~\in \mathbb{Q}/\mathbb{Z},\]
as desired.
\end{proof}

\begin{theorem}[=\Cref{key_cassels_tate_prop_4_sel_intro}] \label{key_cassels_tate_prop_4_sel}
Suppose that $\textup{Sel}^2(A/K)\cap H^1(G,A[2])$ is generated by the class $\mathfrak{c}$. Let $P$ be an admissible pairing on $\textup{Sel}^2(A/K)$. Then there is a quadratic character $\chi:G_K\rightarrow \mu_2$ such that all of the following hold:
\begin{itemize}
\item[(i)] $\textup{rk}_2(A^{\chi}/K)\equiv \textup{rk}_2(A/K) \pmod 2,$
\item[(ii)] $\textup{Sel}^2(A^{\chi}/K)=\textup{Sel}^2(A/K)$ as subgroups of $H^1(K,A[2])$,
\item[(iii)] for all $\mathfrak{a},\mathfrak{b}\in \textup{Sel}^2(A^{\chi}/K)$ we have 
\[\textup{CTP}_{\chi}(\mathfrak{a},\mathfrak{b})=P(\mathfrak{a},\mathfrak{b}).\]
\end{itemize}
\end{theorem}

\begin{proof}
Take $T=\{\mathfrak{a}_1,...,\mathfrak{a}_n\}$ to be a subset of $\textup{Sel}^2(A/K)$ projecting to a basis for the quotient of $\textup{Sel}^2(A/K)$ by the subspace generated by $\mathfrak{c}$. For consistency with \Cref{ssec:ext_classes}, if $\mathfrak{c}\neq 0$ then we define $\mathfrak{a}_{n+1}:=\mathfrak{c}$. Take $T'=\{\mathfrak{a}_{n+1}\}$ if $\mathfrak{c}\neq 0$, and take $T'=\emptyset$ otherwise.  For each $i$, let $a_i:G_K\rightarrow A[2]$ be a $1$-cocycle representing the class of $\mathfrak{a}_i$. Let $\mathcal{S}$ be defined from $T$ and $T'$ as in \Cref{E_n_group}. For each $s=(i,j)\in \mathcal{S}$, fix a $1$-cochain $\gamma_{s}:G_K\rightarrow \mu_2$ such that  $d\gamma_{s}=a_i\cup a_j$. Let $ K_{T,T'}/K$ be the extension defined from these choices as in \Cref{triv_cochain_extension_sec}. Similarly, for each $s\in \mathcal{S}$,  let $\psi_s$  be the function defined from these choices as in \Cref{notat_phiuhhk}.

With this notation set, let $\Sigma_0$ be a finite set of places containing all places dividing $2N_A\infty$  and all places ramified in $K_{T,T'}/K$.  As is possible by \Cref{existence_of_frob_elements_prop} and the Chebotarev density theorem, fix a prime $\mathfrak{p}\notin \Sigma_0$  such that $A[2]^{\textup{Frob}_\mathfrak{p}}=0$, and such that 
\[\quad\quad \quad \psi_{s}(\textup{Frob}_\mathfrak{p})=\textup{CTP}(\mathfrak{a}_i,\mathfrak{a}_j)+P(\mathfrak{a}_i,\mathfrak{a}_j)\quad \quad \quad \textup{ for all }s=(i,j)\in \mathcal{S}.\]
 Write $\Sigma=\Sigma_0\cup \{\mathfrak{p}\}$, and let $\mathcal{C}$ be the subset of $\textup{Gal}(K_{T,T'}/K)$ consisting of those elements $\sigma$ whose image in $G$ lies in the same conjugacy class as $\textup{Frob}_\mathfrak{p}$, but for which 
$\psi_{s}(\sigma)=0$ for all $s\in \mathcal{S}$.  
 By \Cref{existence_of_frob_elements_prop} we see that $\mathcal{C}$ is a non-empty union of conjugacy classes of $\textup{Gal}(K_{T,T'}/K)$. Denote by $\mathcal{P}$ the set of primes  $\mathfrak{p}\notin \Sigma$ for which $\textup{Frob}_\mathfrak{p}$ lies in $\mathcal{C}$. We claim that, using \Cref{approximating_extensions_lemma}, we can find a quadratic character $\chi:G_K\rightarrow \mu_2$ such that:
 \begin{itemize}
 \item $\textup{res}_v(\chi)$ is trivial for all places $v\in \Sigma_0$,
 \item $\textup{res}_\mathfrak{p}(\chi)$ is a given non-trivial ramified quadratic characters of $K_\mathfrak{p}$,
 \item $\chi$ is unramified outside $\Sigma \cup \mathcal{P}$.
 \end{itemize}
 Indeed, denoting by $\alpha$ a non-trivial ramified character of $G_{K_\mathfrak{p}}$, the condition  to be checked  is that 
 \[\textup{inv}_\mathfrak{p}(\alpha \cup \textup{res}_\mathfrak{p}(\beta))=0 \]
 for every homomorphism $\beta:\textup{Gal}(K_{T,T'}/K) \rightarrow \mu_2$ vanishing on $\mathcal{C}$. Take any such $\beta$. Since $\mathfrak{p}$ is chosen to be unramified in $K_{T,T'}/K$, we have 
$\textup{inv}_\mathfrak{p}(\alpha \cup \textup{res}_\mathfrak{p}(\beta))=0$  if and only if $\beta(\textup{Frob}_\mathfrak{p})=1.$
 Now $\beta$ factors through $G$ by \Cref{determinantioN_of_cup_trivi_ext}, and by construction the image of $\textup{Frob}_\mathfrak{p}$ in $G$ is contained in the image of $\mathcal{C}$ in $G$. Hence $\beta$ vanishes  on $\textup{Frob}_\mathfrak{p}$, and we have proven the claim.
 
We now show that $\chi$ satisfies the conditions of the statement. That is satisfies parts (i) and (ii)  follows from  \Cref{selmer_agreement_lemma}. Next, by \Cref{main_ct_comp_prop} we see that, for each $s=(i,j)\in \mathcal{S}$, we have
\[\textup{CTP}_\chi(\mathfrak{a}_i,\mathfrak{a}_j)=P(\mathfrak{a}_i,\mathfrak{a}_j) .\]
Since any admissible pairing on $\textup{Sel}^2(A/K)$ is determined by its values on pairs $(\mathfrak{a}_i,\mathfrak{a}_j)$ for $(i,j)\in \mathcal{S}$, the result  follows from \Cref{ctp_is_admissible} and admissibility of $P$. 
\end{proof}

\section{Parity of $2^\infty$-Selmer ranks} \label{parity section}

 \Cref{key_cassels_tate_prop_4_sel} shows that we can vary the Cassels--Tate pairing whilst keeping   $\textup{Sel}^2(A/K)$ and the parity of $\textup{rk}_2(A/K)$ fixed. For the application to \Cref{main_thm} it will be necessary to vary the parity of $\textup{rk}_2(A/K)$ in a controlled way also.  In this section we prove the results required to do this, building on those of \cite{MR3951582}. We remark that in the case of Jacobians of hyperelliptic curves, one can use \cite[Theorem 1.1]{mor2023} together with an easy root number computation as a substitute for the results of this section.
 
 We do not assume that $A$ satisfies any parts of \Cref{existence_of_frob_elements_prop_1} in this section.
 
 \subsection{Local invariants controlling the parity of $2^\infty$-Selmer ranks}

\begin{notation} \label{local_omega_invs}
 For each place $v \mid 2N_A\infty$, and each quadratic character $\chi:G_{K_v}\rightarrow \mu_2$, let $\Omega_v(\chi)$ be the invariant defined  in \cite[Definition 10.21]{MR3951582} (see also \cite[Definitions 5.15 and Lemma 10.8]{MR3951582}). It takes values in $\{\pm 1\}$. Per \cite[Proposition 5.16 (1)]{MR3951582}, if $\chi$ is trivial then $\Omega_v(\chi)=1$.
\end{notation}
%
%

It is shown in \cite{MR3951582} that the local terms $\Omega_v(\chi)$ govern the parity of $2^\infty$-Selmer ranks of quadratic twists of $A/K$. Specifically, we have the following:

\begin{theorem}\label{parity_2_infinity}
Let $\chi:G_K\rightarrow \mu_2$ be a  quadratic character. Then we have 
\[(-1)^{\textup{rk}_2(A^\chi/K)}=(-1)^{\textup{rk}_2(A/K)}\cdot \prod_{\substack{v\mid 2N_A\infty\\\textup{res}_v(\chi)\textup{ non-trivial}}}\Omega_v(\textup{res}_v(\chi)).\]
\end{theorem}

\begin{proof}
The case when $\chi$ is trivial is clear. Suppose that $\chi$ is non-trivial, and denote by $F/K$ the corresponding quadratic extension. Noting that \[\textup{rk}_2(A/F)=\textup{rk}_2(A/K)+\textup{rk}_2(A^\chi/K),\]
 the result follows from \cite[Theorem 10.20]{MR3951582}.
\end{proof}

The key consequence of \Cref{parity_2_infinity}, as was used in the proof of \Cref{selmer_agreement_lemma}, is that  the parity of $\textup{rk}_2(A^\chi/K)$ depends only on the restriction of $\chi$ to the places dividing $2N_A\infty$. As a result, the precise form of the local terms $\Omega_v(\textup{res}_v(\chi))$ is often not of great importance. However, for our applications to Kummer varieties,  it will be necessary to know conditions under which  $\Omega_v(\chi)$ is non-trival. For this, the following proposition will suffice.

\begin{proposition} \label{omega_prop}
Suppose that $v\nmid 2\infty$ is a place of $K$ such that $A$ has semistable reduction at $v$, and such that the N\'{e}ron component group $\Phi(A/K_v)$ has odd order. Let $\chi$ be the unique non-trivial unramified quadratic character of $G_{K_v}$. Then we have
\[\Omega_v(\chi)=(-1)^{\mathfrak{f}(A/K_v)},\]
where $\mathfrak{f}(A/K_v)$ denotes the conductor exponent of $A$ at $v$.
\end{proposition}

\begin{remark}
Let $F/K_v$ denote the unique unramified extension of degree $2$. By \cite[Corollary 4.6]{MR3860395}, under the assumptions of \Cref{omega_prop}, the local root number $w(A/F_v)$ of $A$ over $F_v$ is equal to $(-1)^{\mathfrak{f}(A/K_v)}$. Thus \Cref{omega_prop} can be interpreted as saying that the local invariant $\Omega_v(\chi)$ is equal to the local root number of $A$ over $F_v$. 
\end{remark}

The proof of \Cref{omega_prop} will occupy the rest of the section.

\subsection{Recollections from \cite{MR3951582}}   We begin by recalling the relevant notions from  \cite{MR3951582}. For the rest of the section, fix a nonarchimedean place $v\nmid 2$ of $K$. Further, we take the following notation:
 
 \begin{notation}
 For a non-trivial quadratic character $\chi:G_{K_v}\rightarrow \mu_2$, take $\mathfrak{g}(A,\lambda,\chi)$ to be the class in the Brauer group of $K_v$ defined in \cite[Definitions 5.15]{MR3951582} (recall that $\lambda$ denotes the fixed principal polarisation on $A$).   Further, denoting by $E_\chi$ the  quadratic extension of $K_v$ corresponding to $\chi$,   we denote by
\[N_{E_\chi/K_v}:A(E_\chi)\longrightarrow A(K_v)\]
the norm (or trace) map sending $x\in A(E_\chi)$ to $\sum_{g\in \textup{Gal}(E_\chi/K_v)}\sigma x$.
 \end{notation}
 
 Per \cite[ Definition 10.21]{MR3951582}, for any quadratic character $\chi:G_{K_v}\rightarrow \mu_2$, we have 
 \begin{equation} \label{om_v_local_def}
 \Omega_v(\chi)=(-1)^{2\textup{inv}_v\mathfrak{g}(A,\lambda,\chi)+\dim A(K_v)/N_{E_\chi/K_v}A(E_\chi)}.
 \end{equation}
 In the case of interest, the contribution from the cokernel of the norm map is trivial.
 \begin{lemma} \label{mazur_towers_lemma}
Let $\chi:G_{K_v}\rightarrow \mu_2$ be the unique unramified quadratic character of $K_v$, and suppose that $\Phi(A/K_v)$ has odd order. Then we have $\dim A(K_v)/N_{E_\chi/K_v}A(E_\chi)=0$.
\end{lemma}

\begin{proof}
This is a consequence of \cite[Prop. 4.2, Prop. 4.3]{MR0444670}.
\end{proof}

%


%

In \Cref{g_defi_long} below, we recall the definition of the  class $\mathfrak{g}(A,\lambda,\chi)$ from \cite[Definitions 5.13 and 5.15]{MR3951582}. We begin with the following notation.

\begin{notation}
Given  a quadratic refinement $q$ of the Weil pairing on $A[2]$,  denote by $c_q:G_{K_v}\rightarrow A[2]$ the corresponding $1$-cocycle representing the class $\mathfrak{c}$, defined as in \cite[Definition 3.1]{MR3951582}. 
Further,  denote by $F_q: G_{K_v} \rightarrow \mu_2$ the function defined in \cite[Definition 5.11]{MR3951582}, viewed as a $\mu_2$-valued $1$-cochain. Per \cite[Remark 5.1]{MR3951582}, we have the cocycle identity
\begin{equation} \label{dcupqc}
dF_q=c_q\cup c_q. 
\end{equation}
\end{notation} 

For the proof of \Cref{omega_prop} we will need the following lemma, which is implicit in \cite{MR3951582}.

\begin{lemma} \label{dickson_hom_lemma}  
Let $q$ be a quadratic refinement of the Weil pairing on $A[2]$, and suppose that $\sigma \in G_{K_v}$ preserves $q$. Then
\[F_q(\sigma)=(-1)^{\dim A[2]^\sigma}.\] 
\end{lemma}

\begin{proof}
By assumption, $\sigma$ acts on $A[2]$ through the orthogonal group $O(q)$ of automorphisms preserving $q$. Denote by $d_q:O(q)\rightarrow \mu_2$ the Dickson homomorphism, whose kernel is the special orthogonal group $SO(q)$. Per \cite[Proposition 3.9]{MR3951582} (and the definition of the map $F_q$; see \cite[Definition 5.11]{MR3951582}), we have $F_q(\sigma)=d_q(\sigma)$.
The result now follows from \cite[Theorem 3]{Dye77}.
\end{proof}

The class $\mathfrak{g}(A,\lambda,\chi)$ is defined as follows.  We refer to \cite[Section 5.8]{MR3951582}  for justification that the choices involved can be made, and that the result is independent of these choices.

\begin{defi} \label{g_defi_long}
Let $\chi:G_{K_v}\rightarrow \mu_2$ be a quadratic character, with corresponding quadratic extension $E_\chi/K_v$. Choose a quadratic refinement $q$  of the Weil pairing on $A[2]$. Choose $P_q\in A(\overline{K_v})$ such that $dP_q=c_q$, and pick $Q_q\in A(\overline{K_v})$ with $2Q_q=P_q$. Set $\rho_q=dQ_q$; this is a $1$-cocycle valued in $A[4]$. Repeating the above with $A$ replaced by $A^\chi$ yields a $1$-cocycle $\rho_{q,\chi}$ valued in $A^\chi[4]$, which we reinterpret as a $1$-cochain valued in $A[4]$  via the isomorphism between $A$ and $A^\chi$ over $E_\chi$. Having done this, the sum $\rho_q+\rho_{\chi,q}$  takes values in $A[2]$, and has coboundary $\chi \cup c_q$. Then
\begin{equation} \label{gqdefi}
(\rho_q+\rho_{\chi,q})\cup c_q - \chi \cup F_q 
\end{equation}
  is a $2$-cocycle valued in $\mu_2$, and  $\mathfrak{g}(A,\lambda,\chi)$ is defined to be its class in $H^2(K_v,\mu_2)$. 
\end{defi}


\subsection{Proof of \Cref{omega_prop}}

We now give  the proof of  \Cref{omega_prop}.
 Fix $\chi$ to be the unique non-trivial unramified quadratic character of $K_v$.  From \eqref{om_v_local_def} and \Cref{mazur_towers_lemma}  we have 
\begin{equation} \label{omega_rsimp_eq}
\Omega_v(\chi)=(-1)^{2\textup{inv}_v \mathfrak{g}(A,\lambda,\chi)}.
\end{equation}
 Since $K_v$ has odd residue characteristic, and since $A$ has semistable reduction at $v$, the inertia group $I_v$ of $K_v$ acts on $A[2]$ through a  finite  cyclic quotient. We thus see from \cite[Proposition 3.6 (a)]{MR2915483} that there is an $I_v$-stable quadratic refinement $q$ of the Weil pairing on $A[2]$, which we fix henceforth.
 
\textbf{Claim:}
 We can make choices so that, in the notation of \Cref{g_defi_long}, the resulting $2$-cochain
\[(\rho_q+\rho_{\chi,q})\cup c_q\]
factors through $\textup{Gal}(K_v^{\textup{ur}}/K_v)$.

\textbf{Proof of claim:} 
Since $q$ is $I_v$-invariant, the cocycle $c_q$ factors through $\textup{Gal}(K_v^{\textup{ur}}/K_v)$. Pick $P_q$  as in \Cref{g_defi_long}. Since $dP_q=c_q$, we see that $P_q$ is defined over $K_v^{\textup{ur}}$.  Further, since $\Phi(A/K_v)$ has odd order, $A(K_v^{\textup{ur}})$ is $2$-divisible. In particular, we can choose $Q_q\in A(K_v^\textup{ur})$ so that $2Q_q=P_q$. The corresponding cocycle $\rho_q=dQ_q$ then factors through $\textup{Gal}(K_v^{\textup{ur}}/K_v)$. Since $A$ and $A^\chi$ are isomorphic over the maximal unramified extension of $K_v$, and since N\'{e}ron models commute with unramified base-change,  $\Phi(A^\chi/K_v)$ has odd order also. Thus we can similarly make choices to ensure that the cochain $\rho_{\chi,q}$ factors through $\textup{Gal}(K_v^{\textup{ur}}/K_v)$, proving the claim. \qed

With choices made as in the claim, the class $\mathfrak{g}(A,\lambda,\chi)$  is by definition represented by the $\mu_2$-valued $2$-cocycle 
 \[ (\rho_q+\rho_{\chi,q})\cup c_q- \chi \cup F_q.\]
 Since   $H^2(\textup{Gal}(K_v^{\textup{ur}}/K_v),\mu_2)$ is trivial (see \cite[Example in Chapter 3]{NSW08}), and since $c_q$ is unramified, we can choose a $1$-cochain $\eta:\textup{Gal}(K_v^{\textup{ur}}/K_v) \rightarrow \mu_2$ such that $d\eta=c_q\cup c_q$.   The difference $F_q-\eta$ is then a quadratic character, and we have 
 \[(\rho_q+\rho_{\chi,q})\cup c_q-\chi \cup F_q=\big((\rho_q+\rho_{\chi,q})\cup c_q+\chi \cup \eta\big) + \chi \cup (F_q-\eta).\]
 The first bracketed term on the right hand side is a $2$-cocycle valued in $\mu_2$ which, by the claim and our choice of $\eta$, factors through $\textup{Gal}(K_v^{\textup{ur}}/K_v)$. Thus it is a $2$-coboundary. Consequently, we have
 \[\textup{inv}_v\mathfrak{g}(A,\lambda,\chi)=\textup{inv}_{v}(\chi \cup (F_q-\eta))=\begin{cases} 0~~&~~\textup{res}_{I_v}(F_q)\textup{ is trivial,}\\ 1/2~~&~~\textup{otherwise.}\end{cases} \]
 Here for the second equality we note that $\chi$ is unramified, and that the functions $F_q$ and $F_q-\eta$ have the same restriction to $I_v$. 
Since $I_v$ preserves $q$,  and acts on $A[2]$ through a finite cyclic quotient, we see from  \Cref{dickson_hom_lemma} that $\textup{res}_{I_v}(F_q)$ is trivial  if and only if $\dim A[2]^{I_v}$ is even.
Thus 
  \[2\textup{inv}_{v}\mathfrak{g}(A,\lambda,\chi)\equiv  \dim A[2]^{I_{v}}\textup{ (mod }2\mathbb{Z}).\]
 Finally, since $A$ has semistable reduction at $v$, and since $\Phi(A/K_v)$ has odd order, we see from  e.g.  \cite[Lemma 4.2]{MR3860395} that
  \[   \dim A[2]^{I_{v}}\equiv \mathfrak{f}(A/K_v) \pmod 2.\]
This completes the proof of  \Cref{omega_prop}. \qed

\section{Proof of \Cref{main_thm}} \label{twisting_4_sel_c}
 
For this section, we assume that $A$ satisfies all three parts of \Cref{existence_of_frob_elements_prop_1}. Let $\mathfrak{a}$ be an element of $H^1(K,A[2])$ satisfying \Cref{assumption:local_solubility}.
 
 \subsection{Field extensions trivialising a given cocycle class} \label{field_ext_trivl_subsec}

Take $T=\{\mathfrak{a}\}$, let $K_T/K$ be the extension defined from $T$ as in \Cref{ssec:cocyclefieldext}, and denote by $\varphi_T$ the corresponding isomorphsim
\begin{equation} \label{iso_semidirect_special_case}
\varphi_T:\textup{Gal}(K_T/K)\stackrel{\sim}{\longrightarrow} A[2]\rtimes G
\end{equation}
afforded by \Cref{iso_to_semi_direct_general}.

\begin{notation}
Let $H$ denote the subgroup of $\textup{Gal}(K_{T}/K)$ consisting of all elements $\sigma$ for which $a(\sigma)=0$. Under the isomorphism \eqref{iso_semidirect_special_case} it corresponds to the subgroup $0\times G$ of $A[2]\rtimes G$. 
\end{notation}

The following lemma  is an analogue of (a special case of ) \cite[Proposition 3.10]{HS16}.

\begin{lemma} \label{2_torsor_trivial_lemma}
Let $\mathcal{C}$ be the set of $\sigma \in \textup{Gal}(K_{T}/K)$ such that $\mathfrak{a}$ is trivial on restriction to the cyclic subgroup generated by $\sigma$. Then we have 
\[\mathcal{C}=\cup_{\sigma \in \textup{Gal}(K_{T}/K)}\sigma H \sigma^{-1}.\] Moreover, given any non-trivial homomorphism $\chi:\textup{Gal}(K_{T}/K) \rightarrow \mu_2$, there is $\sigma \in \mathcal{C}$ such that $\chi(\sigma)=-1$.
\end{lemma}  

\begin{proof} 
Let $Z_{\mathfrak{a}}$ denote the $A[2]$-torsor corresponding to $\mathfrak{a}$; its $\overline{K}$ points are  equal to $A[2]$ as a set, but with $\sigma \in G_K$ acting on $x\in A[2]$ by  
\[  \sigma \cdot x=a(\sigma)+\sigma(x) .\]
This action factors through $\textup{Gal}(K_{T}/K)$, and $H$ is the stabiliser  of the neutral element of $A[2]$ in $\textup{Gal}(K_{T}/K)$. 
The condition that $\mathfrak{a}$ is trivial on restriction to the subgroup generated by  some $\sigma \in \textup{Gal}(K_{T}/K)$ is equivalent to the condition that $\sigma $ fix a $\overline{K}$-point of $Z_{\mathfrak{a}}$. As in \cite[Proof of Corollary 3.7]{HS16}, the fact that \eqref{iso_semidirect_special_case} is an isomorphism means that $\textup{Gal}(K_{T}/K)$ acts transitively on $Z_{\mathfrak{a}}(\overline{K})$, hence the stabilisers of any $2$ points  are conjugate. This gives the claimed description of $\mathcal{C}$. For the second part, we wish to show that $\chi$ does not vanish on $H$. To see this, note first that from \Cref{iso_to_semi_direct_general}, $\chi $ factors through $G$. On the other hand, under the isomorphism \eqref{iso_semidirect_special_case},  $H$ corresponds to the subgroup of $A[2]\rtimes G$  of elements of the form $(0,\sigma)$ for $\sigma\in G$. In particular, any homomorphism that is    trivial on $H$ and factors through $G$ is  necessarily trivial on the whole of $\textup{Gal}(K_{T}/K)$. 
\end{proof}

\subsection{Twists for which $\mathfrak{a}$ lies in the $2$-Selmer group}

Denote by $X$ the generalised Kummer variety associated to $\mathfrak{a}$, defined as in  \cite[Section 6]{HS16}. Similarly, let $Y$ denote the $2$-covering of $A$ corresponding to $\mathfrak{a}$, with corresponding involution $\iota$, and let $\widetilde{Y}$ be as in \cite[Section 6]{HS16}. Thus  $\widetilde{Y}$ is the blow up of $Y$ at the fixed points of $\iota$, the involution $\iota$ extends to an involution $\widetilde{\iota}$ on $\widetilde{Y}$, and we have $X=\widetilde{Y}/\widetilde{\iota}$. Let $p: \widetilde{Y}\rightarrow X$ denote the quotient morphism. For any place $v$ of $K$, and quadratic character $\chi$ of $K_v$, we denote by $Y^\chi$ and $p_\chi:\widetilde{Y}^\chi \rightarrow X^\chi$ the corresponding $K_v$-varieties arising from viewing $\textup{res}_v(\mathfrak{a})$ as an element of $H^1(K_v,A^\chi[2])$. As explained in  \cite[Section 6]{HS16}  we have a $K_v$-isomorphism $X^\chi \cong X$. Recall from \Cref{selmer_conds_intro_notat} the definition of the subspace $ \mathscr{S}_v(A,\chi)\subseteq H^1(K_v,A[2])$. 

\begin{lemma} \label{satisfies_local_selmer}
Let $v$ be a place of $K$ and suppose that $X(K_v)\neq \emptyset$. Then  there is a quadratic character $\chi:G_{K_v}\rightarrow \mu_2$ such that $\textup{res}_v(\mathfrak{a})\in \mathscr{S}_v(A,\chi)$.
\end{lemma}

\begin{proof}
Let $P_v\in X(K_v)$. Then $G_{K_v}$ acts on the set $p^{-1}(P_v)\subseteq \widetilde{Y}(\overline{K_v})$, which has size at most $2$. Denote by $\chi$  the corresponding homomorphism
\[G_{K_v}\longrightarrow \textup{Sym}(p^{-1}(P_v))\hookrightarrow \mu_2.\]
  By construction,   $G_{K_v}$ acts trivially on $p_\chi^{-1}(P_v)$, so that $\widetilde{Y}^\chi(K_v)\neq \emptyset$. Thus also $Y^\chi(K_v)\neq \emptyset$, hence  the image of $\textup{res}_v(\mathfrak{a})$ in $H^1(K_v,A^\chi)$ is trivial. That is,  $\mathfrak{a}\in \mathscr{S}_v(A,\chi)$.
\end{proof}

The following is an analogue of  \cite[Proposition 6.2]{HS16}. That result is established as an application of the fibration method. Here, we instead draw on \Cref{approximating_extensions_lemma}. 

\begin{proposition} \label{prop:has_selmer_character}
Suppose that $X(K_v)\neq \emptyset$ for all places $v$. Then there is a quadratic character $\chi_0: G_K\rightarrow \mu_2$ such that
\[\mathfrak{a}\in \textup{Sel}^2(A^{\chi_0}/K),\quad \textup{rk}_2(A^{\chi_0}/K)\equiv 1 \pmod 2, \quad \textup{ and }\quad  \textup{Sel}^2(A^{\chi_0}/K)\cap H^1(G,A[2])=\left \langle \mathfrak{c}\right \rangle.\]
\end{proposition}

Before proving \Cref{prop:has_selmer_character} we first establish the following preparatory lemma. 

\begin{lemma} \label{exist_of_primes_lemma_inf_image}
Let $\mathfrak{z}\in H^1(G,A[2])$ and suppose that $\mathfrak{z}$ is not in the subspace generated by $\mathfrak{c}$. Let $\Sigma_0$ be a finite set of places of $K$. Then we can find a prime $\mathfrak{p}\notin \Sigma_0$ such that 
\[\textup{res}_{\mathfrak{p}}(\mathfrak{a})=0 \quad \textup{ and }\quad \textup{res}_{\mathfrak{p}}(\mathfrak{z})\neq 0.\] 
\end{lemma}

\begin{proof}
By \Cref{existence_of_frob_elements_prop_1} (C) there is $g\in G$ such that $\mathfrak{z}$ has non-trivial restriction to the cyclic subgroup generated by $g$. Let $K_T$ be as in \Cref{field_ext_trivl_subsec}, so that both $\mathfrak{z}$ and $\mathfrak{a}$ factor throught $\textup{Gal}(K_T/K)$. Let $\sigma$ by the element of $\textup{Gal}(K_T/K)$ mapping under the isomorphism \eqref{iso_semidirect_special_case} to $(0,g)\in A[2]\rtimes G$.  By construction, the restriction of $\mathfrak{a}$ to the subgroup generated by $\sigma$ is trivial, while the restriction of $\mathfrak{z}$ to the subgroup generated by $\sigma$ is non-trivial. Let $\mathfrak{p}\notin \Sigma_0$ be any  prime  unramified in $K_T/K$ for which the Frobenius element $\textup{Frob}_\mathfrak{p}\in \textup{Gal}(K_T/K)$ lies in the same conjugacy class as $\sigma$. Then $\mathfrak{p}_0$ satisfies the requirements of the statement.
\end{proof}

\begin{proof}[Proof of \Cref{prop:has_selmer_character}]
Let $\mathfrak{p}_0$ be as in \Cref{assumption:local_solubility}. Denote by $\Sigma_0$ a finite set of places of $K$ containing all places dividing $2N_A\infty$ (hence in particular  $\mathfrak{p}_0$), and all primes $\mathfrak{p}$ for which $\textup{res}_\mathfrak{p}(\mathfrak{a})$ is ramified. For each place $v\in \Sigma_0 \setminus \{\mathfrak{p}_0\}$, \Cref{satisfies_local_selmer} provides a quadratic character $\chi_v:G_{K_v}\rightarrow \mu_2$  such that $\textup{res}_v(\mathfrak{a})\in \mathscr{S}_v(A,\chi_v)$.  With the local invariants $\Omega_v(\chi_v)$ as in \Cref{local_omega_invs}, set
\[\kappa=(-1)^{\textup{rk}_2(A/K)}\cdot\prod_{v\in \Sigma_0 \setminus \{\mathfrak{p}_0\}}\Omega_v(\chi_v).\]
If $\kappa=-1$ take $\chi_{\mathfrak{p}_0}$ to be the trivial character of $G_{K_{\mathfrak{p}_0}}$, whilst if $\kappa=1$ take $\chi_{\mathfrak{p}_0}$ to be the unique non-trivial unramified quadratic character of $G_{K_{\mathfrak{p}_0}}$.  

Next, as is possible by \Cref{exist_of_primes_lemma_inf_image},  for each of the finitely many classes $\mathfrak{z}\in H^1(G,A[2])\setminus \left \langle \mathfrak{c}\right \rangle$, fix a prime $\mathfrak{p}_{\mathfrak{z}}\notin \Sigma_0$ such that $\textup{res}_{\mathfrak{p}_\mathfrak{z}}(\mathfrak{a})=0$ and $\textup{res}_{\mathfrak{p}_{\mathfrak{z}}}(\mathfrak{z})\neq0$. For each such $\mathfrak{z}$ fix also a ramified quadratic character $\chi_{\mathfrak{z}}$ of $G_{K_{\mathfrak{p}_\mathfrak{z}}}$. Take $\Sigma$ to be the union of $\Sigma_0$ with the set of all    $\mathfrak{p}_\mathfrak{z}$.

By combining \Cref{approximating_extensions_lemma} with \Cref{2_torsor_trivial_lemma}, we see that there is a quadratic character $\chi_0: G_K\rightarrow \mu_2$ such that $\textup{res}_v(\chi_0)=\chi_v$ for each $v\in \Sigma$, and such that, for all primes $\mathfrak{p}\notin \Sigma$ at which $\chi_0$ is ramified, we have $\textup{res}_v(\mathfrak{a})=0$. We claim that $\chi_0$ satisfies the conditions of the statement. To see this, note first that  from \Cref{parity_2_infinity} and \Cref{omega_prop} we have $\textup{rk}_2(A^{\chi_0}/K)$  odd. 

Next, we show that $\mathfrak{a}$ lies in $\textup{Sel}^2(A^{\chi_0}/K)$. Let $v$ be a place of $K$. We wish to show that $\textup{res}_v(\mathfrak{a})$ lies in $\mathscr{S}_v(A,\textup{res}_v(\chi_0))$. First, if $v\in \Sigma \setminus \{\mathfrak{p}_0\}$, or if $v\notin \Sigma$ and $\textup{res}_v(\chi_0)$ is ramified, then this holds by construction. Finally, if $v\notin \Sigma$ and $\chi_0$ is unramified, or if $v=\mathfrak{p}_0$, then we conclude from \Cref{selmer_conds_lemma}(ii) and the assumption that $\mathfrak{a}$ is unramified at $v$.

Finally, given any $\mathfrak{z}\in H^1(G,A[2])\setminus \left \langle \mathfrak{c}\right \rangle$, we wish to show that $\mathfrak{z}\notin \textup{Sel}^2(A^{\chi_0}/K)$.  To see this, note that $\mathfrak{z}$ is unramified at $\mathfrak{p}_{\mathfrak{z}}$ (since $\mathfrak{p}_\mathfrak{z}\notin \Sigma_0$) and that $\textup{res}_{\mathfrak{p}_\mathfrak{z}}(\mathfrak{z})\neq 0$ by construction. However, since we have chosen $\chi_0$ to ramify at $\mathfrak{p}_\mathfrak{z}$, \Cref{selmer_conds_lemma}(iii) shows that $\mathscr{S}_v(A,\textup{res}_v(\chi_0))\cap H^1_{\textup{ur}}(K_{\mathfrak{p}_\mathfrak{z}},A[2])=0$. 
\end{proof}

  \begin{cor} \label{cnonzeroshadivcor}
There is a quadratic character $\chi:G_K\rightarrow \mu_2$ such that $\textup{rk}_2(A^\chi/K)=1$, such that $\mathfrak{a}\in \textup{Sel}^2(A^\chi/K)$, and such that the image of $\mathfrak{a}$ in $\Sha(A^\chi/K)[2^\infty]$ is divisible. In particular, if $\Sha(A^\chi/K)[2^\infty]$ is finite, then $\mathfrak{a}$ has trivial image in $\Sha(A^\chi/K)[2^\infty]$. 
 \end{cor}
 
 \begin{proof}
Let $\chi_0$ be as in \Cref{prop:has_selmer_character}. We will apply \Cref{key_cassels_tate_prop_4_sel} with $A^{\chi_0}$ in place of $A$. Specifically, from that result we see that we can find a quadratic character $\chi$ such that
\[\textup{rk}_2(A^\chi/K)\equiv 1 ~(\textup{mod } 2) \quad \textup{ and } \quad  \quad \textup{Sel}^2(A^{\chi}/K)=\textup{Sel}^2(A^{\chi_0}/K),\]
and such that $\textup{CTP}_\chi$ is any bilinear pairing $P$ on $\textup{Sel}^2(A^{\chi_0}/K)$ we desire, subject only to the conditions
\[2P(\mathfrak{c},\mathfrak{c})\equiv 1+\dim \textup{Sel}^2(A^{\chi_0}/K)\]
and 
\[\quad \quad\quad \quad\quad \quad \quad P(\mathfrak{b},\mathfrak{b})=P(\mathfrak{b},\mathfrak{c})\quad \quad \textup{ for all }\mathfrak{b}\in \textup{Sel}^2(A^{\chi_0}/K).\]
Note that when $\mathfrak{c}=0$,  \cite[Theorem 8]{MR1740984} and the fact that $\textup{rk}_2(A^{\chi_0}/K)$ is odd forces $\dim \textup{Sel}^2(A^{\chi_0}/K)$ to be odd also, so the first condition on $P$ is automatically satisfied.
It is now an easy exercise to show that there exists such a pairing $P$ for which the given class $\mathfrak{a}$ is the unique non-trivial element in its kernel. Thus we can find $\chi$ such that $\textup{rk}_2(A^\chi/K)$ is odd, and such that $\mathfrak{a}$ is the unique non-trivial element of $\textup{Sel}^2(A^\chi/K)$ lifting to $\textup{Sel}^4(A^\chi/K)$. For any such $\chi$ we  have both that   $\textup{rk}_2(A^\chi/K)=1$ and  that  $\mathfrak{a}$ is divisible in $\Sha(A^\chi/K)[2^\infty]$.
 \end{proof}
 
 \begin{proof}[Proof of \Cref{main_thm}]
 Let $\chi$ be as in \Cref{cnonzeroshadivcor}. Let $Y^\chi$ denote the $2$-covering for $A^\chi$ arising from viewing $\mathfrak{a}$ as an element of $H^1(K,A^\chi[2])$.   Under the assumption that the Shafarevich--Tate group of $A^\chi$ is finite, $\mathfrak{a}$ has trivial image in $\Sha(A^\chi/K)$. In particular, the corresponding $2$-covering $Y^\chi$, hence also the corresponding Kummer variety $X^\chi\cong X$ has a $K$-point. The same argument as given in the final paragraph of the proof of \cite[Proof of Theorem 2.2]{HS16} now shows  moreover that the $K$-points are Zariski dense in $X$ (as a consequence of part (A) of \Cref{existence_of_frob_elements_prop_1}). 
 \end{proof} 

 \section{Example: Jacobians of hyperelliptic curves} \label{sec:specialising_to_hyp_curves}
 
In this section we specialise to Jacobians of hyperelliptic curves and explain how to deduce \Cref{thm:hyp_curves_main} from \Cref{main_thm}. Although much of the material in this section is standard, we include it for completeness.
 
Let $d\geq 2$ be an integer and let $f(x)\in K[x]$ be a squarefree polynomial of degree $2d+2$. Let $C:y^2=f(x)$ be the corresponding genus $d$ hyperelliptic curve and take $A$ to be the canonically principally polarised Jacobian of $C$. Since $d$ is even, there are $2$ points $\infty^{\pm}$ at infinity on $C$. 

\subsection{The Galois module $A[2]$}
As is well known, the $G_K$-module $A[2]$ admits the following explicit description. Denote by $\mathcal{R}$ the set of roots of $f(x)$ in $\overline{K}$ and let $\mathbb{F}_2^\mathcal{R}$ be the permutation module over $\mathbb{F}_2$ on the elements of $\mathcal{R}$. 
 We identify $\mathbb{F}_2^{\mathcal{R}}$ with the power set of $\mathcal{R}$ via the map sending a subset to its characteristic function. Having made this identification, define the $G_K$-invariant $\mathbb{F}_2$-valued pairing $e$ on  $\mathbb{F}_2^{\mathcal{R}}$ by  setting
\[e(S,T)=\#(S \cap T) \pmod 2.\]
Define $V_0$ to be the subspace of $\mathbb{F}_2^{\mathcal{R}}$ consisting of even sized subsets of $\mathcal{R}$, and define $V$ to be the quotient of $V_0$ by the $1$-dimensional subspace generated by the set $\mathcal{R}$. Then $V$ is an $\mathbb{F}_2$-vector space of dimension $2d$, and the pairing $e$ descends to a non-degenerate alternating pairing on $V$. Per \cite[Section 5.2]{Dol12}, there is an   isomorphism of $G_K$-modules  
\begin{equation} \label{explicit_2_tors}
V\cong A[2]
\end{equation}
identifying  $e$ with the Weil pairing. It is realised by the map sending an even sized subset $S\subseteq \mathcal{R}$ to the class of the divisor
 \[\sum_{r\in S}(r,0) - (\tfrac{1}{2}\#S )\cdot \big(\infty^++\infty^{-}\big).  \] 
 In particular, the  $G_K$-action on $A[2]$ factors through the Galois group $\textup{Gal}(f)$ of $f(x)$. Using that $d\geq 2$, one checks that the induced action of $\textup{Gal}(f)$ on $A[2]$ is faithful, hence $G:=\textup{Gal}(K(A[2])/K)$ is equal to $\textup{Gal}(f)$. Thus $G$ can naturally be viewed as a subgroup of the symmetric group $S_{2d+2}$.





\begin{lemma} \label{lem:yp+satisfies_all_conds}
Suppose that $G=\textup{Gal}(f)$ is isomorphic to $S_{2d+2}$. Then $A$ satisfies parts (A) and (B) of \Cref{existence_of_frob_elements_prop_1}.
\end{lemma}

\begin{proof}
This is shown in \cite[Lemma 2.1]{HS16} when $G$ is the full stabiliser in $S_{2d+2}$ of a root of $f(x)$ (hence isomorphic to $S_{2d+1}$). Thus it holds for $G\cong S_{2d+2}$ also.
\end{proof}

\subsection{The classes $\mathfrak{w}$ and $\mathfrak{c}$}
Let $\mathcal{W}$ denote the collection of odd-sized subsets of $\mathcal{R}$, modulo the equivalence relation identifying a subset with its complement. This carries a natural $G_K$-action, and is a principal homogeneous space for $V$ with action given by symmetric difference. Thus associated to $\mathcal{W}$ is a class $\mathfrak{w}$ in $H^1(K,V)=H^1(K,A[2])$, which arises via inflation from $H^1(G,A[2])$. Equivalently, from \eqref{explicit_2_tors} we have a short exact sequence of $G$-modules
\begin{equation} \label{definined_2_tors_hyp}
0\longrightarrow A[2] \longrightarrow \mathbb{F}_2^{\mathcal{R}}/\left \langle \mathcal{R}\right \rangle \longrightarrow \mathbb{F}_2\longrightarrow 0,
\end{equation}
the map $\mathbb{F}_2^{\mathcal{R}}/\left \langle \mathcal{R}\right \rangle \rightarrow \mathbb{F}_2$ sending a subset of $\mathcal{R}$ to its size modulo $2$. Then $\mathfrak{w}$ is the image of $1\in \mathbb{F}_2$ under the connecting map $H^0(K,\mathbb{F}_2)\rightarrow H^1(K,A[2])$ associated to this sequence. 

\begin{proposition}[\cite{Pol71}, Theorem 5.2] \label{prop:when_iso_to_S}
Suppose that $G=\textup{Gal}(f)$ is isomorphic to $S_{2d+2}$. Then $H^1(G,A[2])$ is isomorphic to $\mathbb{F}_2$, generated by the class $\mathfrak{w}$. 
\end{proposition}



The class $\mathfrak{c}$ of \Cref{ssec:class_c} is related to $\mathfrak{w}$ as follows.

\begin{proposition} \label{prop@c_in_hyp+case}
 If $d$ is odd then the class $\mathfrak{c}$ is trivial, while if $d$ is even we have $\mathfrak{c}=\mathfrak{w}$.  Now suppose that $G=\textup{Gal}(f)$ is isomorphic to $S_{2d+2}$. If $d$ is odd then 
\begin{equation} \label{ugly_intersection}
\bigcap_{g\in G}\ker\Big(H^1(G,A[2])\stackrel{\textup{res}}{\longrightarrow}  H^1(\left \langle g\right \rangle, A[2]) \Big) =0.
\end{equation}
If $d$ is even then  $\mathfrak{c}$ is non-zero and generates $H^1(G,A[2])$.
\end{proposition}
 
\begin{proof}
As we explain below, this follows readily from the standard description of quadratic refinements of the pairing $e$, as detailed, for example, in \cite[Section 5.2.3]{Dol12}.

Suppose first that $d$ is odd. Then the map $q:V_0\rightarrow \mathbb{F}_2$ given by $q(S)=\tfrac{1}{2}\#S \pmod 2$ descends to a map $V\rightarrow \mathbb{F}_2$, giving a  $G$-invariant quadratic refinement of $e$.  In particular, the class $\mathfrak{c}$ is trivial. Now suppose that $G\cong S_{2d+2}$. Since $H^1(G,A[2])$ is generated by the class $\mathfrak{w}$, to show that the intersection \eqref{ugly_intersection} is trivial it suffices to show that there exists $g\in G$ such that $\mathfrak{w}$ has non-trivial restriction to $H^1(\left \langle g \right \rangle,A[2])$.  Write $n=2d+2$ and fix a bijection between $\mathcal{R}$ and $[n]=\{1,2,...,n\}$, via which we identify $G=\textup{Gal}(f)$ with $S_n$. Take $g$ to be a product of disjoint $(d+1)$-cycles, say 
\[ g=(1~3~...~2d+1)(2~4~...~2d+2).\] Then $H^0(\left \langle g \right \rangle, V)$ is $2$-dimensional, generated by  the (class of the) set consisting of all odd numbers in $[n]$, and the (class of the) set consisting of all elements  $r\in [n]$ with $r\equiv 0,1 \pmod 4$.  Since both sets have size $d+1\equiv 0 \pmod 2$, we see from the sequence \eqref{definined_2_tors_hyp} that $\mathfrak{w}$ is not in the kernel of restriction to $H^1(\left \langle g \right \rangle,A[2])$. 


Now suppose that $d$ is even. Let $T\subseteq [n]$ be any odd sized subset. Then the map  $q_T:V\rightarrow \mathbb{F}_2$ given by \[q_T(S)=\#S\cap T+ \tfrac{1}{2}\#S \pmod 2\]
is well defined, and gives a quadratic refinement of $e$. From this, one checks that the map $T\mapsto q_T$ identifies $\mathcal{W}$ with the set of quadratic refinements of $e$, so that $\mathfrak{c}=\mathfrak{w}$. The remaining statement follows from \Cref{prop:when_iso_to_S}. 
\end{proof}
 
\subsection{Proof of \Cref{thm:hyp_curves_main}}

We can now deduce \Cref{thm:hyp_curves_main} from \Cref{main_thm}.

\begin{proof}[Proof of \Cref{thm:hyp_curves_main}] 
Let $\mathfrak{a}\in H^1(K,A[2])$ be as in the statement, and let $X$ denote the corresponding generalised Kummer variety. 

If $\mathfrak{a}\in \{0,\mathfrak{w}\}$ then $X$ contains a line defined over $K$, hence $X(K)\neq \emptyset$. Indeed, if $\mathfrak{a}=0$ then $X$ is the blow up of $A/\{\pm 1\}$ at the image of $A[2]$, and the exceptional fibre over the image of $0\in A(K)$ gives the sought line. If $\mathfrak{a}=\mathfrak{w}$ then   $X$ is the  minimal desingularisation of the quotient of $\textup{Pic}^{1}_{C/K}$ by the involution induced from the hyperelliptic involution $\iota$ on $C$. The natural embedding of $C$ into $\textup{Pic}^{1}_{C/K}$  induces an embedding $C/\iota \cong \mathbb{P}^1  \rightarrow X$.   

Suppose that $\mathfrak{a}\notin \{0,\mathfrak{w}\}$. By \Cref{prop:when_iso_to_S} we have $\mathfrak{a}\notin H^1(G,A[2])$. Combining \Cref{lem:yp+satisfies_all_conds} with \Cref{prop@c_in_hyp+case} we see that $A$ satisfies all three parts of \Cref{existence_of_frob_elements_prop_1}. Finally, the assumptions on $C$  ensure that, at $\mathfrak{p}_0$, $C$ has semistable reduction, trivial N\'{e}ron component group and conductor exponent $1$  (see e.g. \cite{hyperuser}). We can now conclude from  \Cref{main_thm}.
\end{proof}

\begin{remark} \label{rem: zariski_dense}
In the setting of \Cref{thm:hyp_curves_main}, the proof  shows that, if  $\mathfrak{a}\notin \{0,\mathfrak{w}\}$ and the Kummer variety $X$ associated to $\mathfrak{a}$ is everywhere locally soluble, then (we can apply   \Cref{main_thm} to conclude that) it has a Zariski-dense set of $K$-points.
\end{remark}
 
 \subsection{The case $d=2$} \label{Kummer_intersection_of_quadrics}
 
 We conclude by specialising to the case $d=2$, where \Cref{thm:hyp_curves_main} can be made more explicit. Assume from now on that $f(x)$ has degree $6$ (thus $C$ has genus $2$).
 
 Write $F=K[x]/(f(x))$ and denote by $\theta$ the image of $x$ in $F$. As shown in \cite{MR1465369} (see Equation (12) in particular), there is a canonical isomorphism
 \begin{equation} \label{poonen_cohom}
 H^1(K,A[2])^{\cup \mathfrak{c}=0}/\left \langle \mathfrak{c} \right \rangle \cong \ker\big(F^{\times}/(K^\times F^{\times 2})\stackrel{\textup{N}_{F/K}}{\longrightarrow}K^{\times}/K^{\times 2}\big),
 \end{equation}
 where $N_{F/K}:F^{\times}\rightarrow K^{\times}$ is the norm map and $H^1(K,A[2])^{\cup \mathfrak{c}=0}$ denotes the subgroup of $H^1(K,A[2])$ consisting of elements $\mathfrak{b}$ for which $\mathfrak{b}\cup \mathfrak{c}=0$ in $H^2(K,\mu_2)$.
 
 \begin{lemma} \label{cup_with_c}
Let $\mathfrak{a}\in H^1(K,A[2])$ and denote by $X$ the corresponding generalised Kummer surface. If $X(K_v)\neq \emptyset $ for all places $v$, then $\mathfrak{a}\in H^1(K,A[2])^{\cup \mathfrak{c}=0}$.  
 \end{lemma}
 
 \begin{proof}
 Let $v$ be a place of $K$. Since $X(K_v)\neq \emptyset$ we have, as in \Cref{satisfies_local_selmer},  that $\textup{res}_v(\mathfrak{a})\in \mathscr{S}_v(A,\chi)$ for some local quadratic character $\chi:G_{K_v}\rightarrow \mu_2$. Since $\mathscr{S}_v(A,\chi)$ is isotropic for the local Tate pairing  \eqref{eq:local_tate}, this gives
\[0=\textup{res}_v(\mathfrak{a})\cup \textup{res}_v(\mathfrak{a})\stackrel{}{=}\textup{res}_v(\mathfrak{a})\cup \textup{res}_v(\mathfrak{c}),\]
where the second equality is \cite[Theorem 3.4]{MR2915483}. Thus $\textup{res}_v(\mathfrak{a}\cup \mathfrak{c})=0$ for all places $v$, hence $\mathfrak{a}\cup \mathfrak{c}=0$ in $H^1(K,\mu_2)$ by the Albert–Brauer–Hasse–Noether theorem.
 \end{proof}
 
 By \Cref{cup_with_c},  \Cref{thm:hyp_curves_main} is only interesting when applied to classes in $ H^1(K,A[2])^{\cup \mathfrak{c}=0}$. When this is the case,  the corresponding Kummer surface admits a description as an intersection of three quadrics in $\mathbb{P}^5$.
 
 \begin{lemma} \label{prop:kummer_equations}
 Let $\lambda \in  F^\times$ be such that $N_{F/K}(\lambda)\in K^{\times 2}$. Let   $\mathfrak{a}\in H^1(K,A[2])^{\cup \mathfrak{c}=0}$ be an element whose image in $H^1(K,A[2])^{\cup \mathfrak{c}=0}/\left \langle \mathfrak{c}\right \rangle $ corresponds to $\lambda$ via \eqref{poonen_cohom}, and let $X$ be the generalised Kummer surface associated to $\mathfrak{a}$. Then $X$ is isomorphic to the   variety defined by the three quadratic equations
 \begin{equation}\label{Kummer_eqs}
 \textup{Tr}_{F/K}\big(\lambda u^2/f'(\theta)\big)=\textup{Tr}_{F/K}\big(\lambda \theta u^2/f'(\theta)\big)=\textup{Tr}_{F/K}\big(\lambda  \theta^2 u^2/f'(\theta)\big)=0
 \end{equation}
 in $\mathbb{P}(F)\cong \mathbb{P}^5.$
 Here, $u$ is an $F$-variable and $\textup{Tr}_{F/K}$ is the trace map. 
 \end{lemma}
 
 \begin{proof}
 By \cite[Proposition 2.6]{Fisher_Yan23} (see Equation (27) in particular), the surface in $\mathbb{P}(F)$ defined by the $3$ quadratic forms in the statement agrees with the surface $V_\lambda$ defined in \cite[Section 4]{Flynn_testa_van_luijk}. That this is isomorphic to the generalised Kummer variety $X$ is shown in the discussion following \cite[Theorem 7.4]{Flynn_testa_van_luijk}. See also  \cite[Section 2.6]{Fisher_Yan23}.
 \end{proof}
 
 \begin{remark}
In the statement of \Cref{prop:kummer_equations}, the equations for $X$ in $\mathbb{P}(F)$  visibly depend only on the class of $\mathfrak{a}$ in $H^1(K,A[2])/\left \langle \mathfrak{c}\right \rangle$. For a direct proof that the Kummer varieties associated to the class $\mathfrak{c}$ and to the zero-class are isomorphic, see \cite[Theorem 3.5]{GHC22}  (this fact goes back to work of Cassels--Flynn \cite[Chapter 16]{CasselsFlynn96}). Twisting this isomorphism by  $\mathfrak{a}$ then shows that the Kummer varieties associated to $\mathfrak{a}$ and $\mathfrak{a}+\mathfrak{c}$ are isomorphic. (We are not aware of this phenomenon extending beyond genus $2$ Jacobians.)
 \end{remark}
 
 We can use \Cref{prop:kummer_equations} to give a more explicit version of \Cref{thm:hyp_curves_main} in the case $d=2$. This is a direct analogue of
 \cite[Theorem B]{HS16}, which treats the case when $f(x)$ factors as the product of a linear polynomial and an irreducible quintic. As usual, we denote by $c_f$ the leading coefficient of $f(x)$ and denote by $\textup{disc}(f)$ the discriminant of $f(x)$. 
 
 \begin{theorem}
Suppose that $f(x)$ has Galois group $S_6$. Let 
$\lambda \in F^{\times}$  be such that $N_{F/K}(\lambda)\in K^{\times 2}$,
 and let $X$ be the variety in $\mathbb{P}(F)\cong \mathbb{P}^5$ cut out by the three quadratic froms \eqref{Kummer_eqs}. Suppose  there is an odd prime $\mathfrak{p}_0$ of $K$ such that: 
 \begin{itemize}
\item[(i)] $f(x)$ is integral at $\mathfrak{p}_0$,  
 $\textup{ord}_{\mathfrak{p}_0}\textup{disc}(f)=1$, $\textup{ord}_{\mathfrak{p}_0}(c_f)=0,$   and
\item[(ii)] there is $t\in K^{\times}$ such that $\textup{ord}_\mathfrak{q}(t\lambda)\equiv 0 \pmod 2$ for each prime $\mathfrak{q}$ of $F$ lying over $\mathfrak{p}_0$.
\end{itemize}
Assume  that the $2$-primary part of the Shafarevich--Tate group of every quadratic twist of $A$ with $2^\infty$-Selmer rank $1$ is finite.  Then $X$ satisfies the Hasse principle. If moreover  $\lambda \notin K^{\times}F^{\times 2}$, then $X$ has a Zariski-dense set of $K$-points.
 \end{theorem}
 
 \begin{proof}
  Let   $\mathfrak{a}\in H^1(K,A[2])$ be an element whose image modulo $\mathfrak{c}$ corresponds to $\lambda$ via \eqref{poonen_cohom}. Note that if $\lambda \notin K^{\times} F^{\times 2}$ then $\mathfrak{a}\notin \{0,\mathfrak{c}\}$. By \Cref{prop:kummer_equations} and  \Cref{rem: zariski_dense},  it suffices to show that condition (ii) of the statement implies that one of $\mathfrak{a}$, $\mathfrak{a}+\mathfrak{c}$ is unramified at $\mathfrak{p}_0$.\footnote{In fact, as was shown in the proof of \Cref{omega_prop}, the class $\mathfrak{c}$ is unramified at $\mathfrak{p}_0$. Thus $\mathfrak{a}$ is unramified at $\mathfrak{p}_0$ if and only if $\mathfrak{a}+\mathfrak{c}$ is.}
  
To make the map in \eqref{poonen_cohom} explicit, let $m\in K^{\times}$ be such that $N_{F/K}(\lambda)=m^2$. Define the $G_K$-set 
\[T_{\lambda,m}=\Big\{z\in F_{\overline{K}}^{\times} ~~\colon~~z^2=\lambda, ~N_{F_{\overline{K}}/\overline{K}}(z)=m\Big\}/\{\pm 1\}.\] 
Here we write $F_{\overline{K}}$ as shorthand for $F\otimes_K \overline{K}$, and the action of $-1$ is given by sending $z$ to $-z$. 
Multiplication in $F_{\overline{K}}^{\times}$ makes $T_{\lambda,m}$ into a principal homogeneous space for the $G_K$-module $\mu_2(F_{\overline{K}})_{N=1}/\mu_2(K)$, where $\mu_2(F_{\overline{K}})_{N=1}=\{\epsilon \in  F_{\overline{K}}^\times\colon \epsilon^2=1, N_{F_{\overline{K}}/\overline{K}}(\epsilon)=1\}.$  Now \eqref{explicit_2_tors}  gives a canonical isomorphism of $G_K$-modules \[A[2]\cong  \mu_2(F_{\overline{K}})_{N=1}/\mu_2(K),\] so the two $G_K$-sets $T_{\lambda, \pm m}$ give rise to $2$ classes in $H^1(K,A[2])$. By \cite[Proposition 2.6]{Flynn_testa_van_luijk}, these are precisely the classes $\mathfrak{a}$ and $\mathfrak{a}+\mathfrak{c}$. 

To make contact with the proof of \cite[Theorem B]{HS16}, let $\mathcal{T}_{\lambda,\pm m}$ be the torsor under the group scheme $A[2]$ whose set of $\overline{K}$-points is $T_{\lambda,\pm m}$. Then one of $\mathfrak{a}$ and $\mathfrak{a}+\mathfrak{c}$ is unramified at $\mathfrak{p}_0$ if and only if the $(\textup{Res}_{F/K}\mu_2)/\mu_2$-torsor
$\mathcal{T}_{\lambda,m}\sqcup \mathcal{T}_{\lambda,-m}$
has a point over $K_{\mathfrak{p}_0}^{\textup{ur}}$, where $\textup{Res}_{F/K}$ denotes Weil restriction. The torsor $\mathcal{T}_{\lambda,m}\sqcup \mathcal{T}_{\lambda,-m}$ is isomorphic to the subscheme of $\textup{Res}_{F/K}(\mathbb{G}_m)/\{\pm 1\}$ given by the equation $z^2=\lambda$. Now arguing exactly as in \cite[Proof of Theorem B]{HS16} gives the result.
 \end{proof}

\end{document}